\documentclass[reqno]{article}
\usepackage[ansinew]{inputenc}
\usepackage{latexsym}
\usepackage{amsfonts}
\usepackage{amsmath,mathrsfs,amssymb,amsthm,amscd}
\allowdisplaybreaks
\usepackage{longtable}
\usepackage{multirow}
\usepackage{color}
\usepackage{url}
\usepackage{hyperref}
\usepackage{cite}
\newtheorem{theo}{Theorem}[subsection]
\newtheorem{lem}[theo]{Lemma}
\newtheorem{prop}[theo]{Proposition}
\newtheorem{coro}[theo]{Corollary}

\newcommand{\w}{\omega}
\newcommand{\E}{\mathcal{E}}
\newcommand{\U}{\mathcal{U}}
\newcommand{\I}{\mathcal{I}}
\newcommand{\J}{\mathcal{J}}
\newcommand{\g}{\gamma}
\newcommand{\al}{\alpha}

\newcommand{\de}{\delta}
\newcommand{\be}{\beta}

\newcommand{\ann}{\mathrm{ ann}}
\def\thebibliography#1{\section*{{\large Bibliographie}\markboth
 {REFERENCES}{REFERENCES}}\list
 {[\arabic{enumi}]}{\settowidth\labelwidth{[#1]}\leftmargin\labelwidth
 \advance\leftmargin\labelsep
 \usecounter{enumi}}
 \def\newblock{\hskip .11em plus .33em minus -.07em}
 \sloppy
 \sfcode`\.=1000\relax}

\def\qed{\relax\ifmmode\hskip2em \Box\else\unskip\nobreak\hskip1em \hfill$\Box$\fi}

\title{Power-associative evolution algebras}

\author{
Moussa OUATTARA\thanks{\texttt{ouatt\_ken@yahoo.fr}}\quad\quad  Souleymane SAVADOGO\thanks{\texttt{sara01souley@yahoo.fr}}\\
D\'epartement de Math\'ematiques et Informatique\\
Universit\'e Ouaga I Pr Joseph KI-ZERBO\\
03 BP 7021 Ouagadougou 03\\
Burkina Faso
}
\date{}
\begin{document}
\maketitle

\begin{abstract}
 The paper is devoted to the study of evolution algebras that are power-associative algebras.

We give the Wedderburn decomposition of evolution algebras that are power-associative algebras and we prove that these algebras are Jordan algebras. Finally, we use this decomposition to classify these algebras up to dimension six.

\medskip 2010 Mathematics Subject Classification~: Primary 17D92, 17A05, Secondary 17D99, 17A60

\end{abstract}
\bigskip\bigskip

\section{Introduction}

The notion of evolution algebras was introduced in 2006 by J.P. Tian and P. Vojt\v echovsk\'y (\cite{Tian2006}). In 2008, J.P. Tian laid the foundations of this algebra in his monograph (\cite{Tian2008}).

In (\cite{Tian2006}, Theorem~1.5), the authors show that evolution algebras are commutative (hence flexible) not necessarily power-associative (nor associative).

Let $\E$ be  $n$-dimensional algebra over a commutative field $F.$ We say that $\E$ is \textit{$n$-dimensional evolution algebra}, if it admits a basis $B = \{e_{i}$ ; $1 \leq i \leq n\}$ such that:
\begin{eqnarray}\label{Eq1}
e_{i}^{2} = \sum_{k=1}^{n} a_{ik}e_{k}, \textrm{ for all $1\leq i\leq n$, the other products being void}.
\end{eqnarray}
The basis $B$ is called \textit{natural basis} of $\E$ and the matrix $(a_{ij})$ is the matrix of structural constants of $\E,$ relative to the natural basis $B.$

In (\cite{Casado2016}, p.127), the authors exhibit a necessary condition for an evolution algebra to be power-associative. They show in particular that the only power-associative evolution algebras are such that $a_{ii}^2=a_{ii}$ ($i=1,\ldots,n$). However, this condition is not sufficient. It is verified for all nil evolution algebras  of the fact that for any $i=1,\ldots,n$, $a_{ii}=0$ (\cite[Proof of Theorem~2.2]{Casas2014}).

In \cite[Example~4.8]{camacho}, for $k = n = 4,$ we find a power-associative algebra which is not associative. Indeed, it is isomorphic to $N_{4,6}$ (see Table~1). Except the type $[1,1,1]$, all algebras in \cite[Table~1]{Elduque2016} are nil associative evolution algebras. In \cite[Table~2]{Elduque2016}, the first line is associative, while the third is power-associative. In \cite[Table~1]{Casado}, the fifth, sixth and seventh algebras, of dimension 3, are associative. They are respectively isomorphic to $\E_{3,4}$, $N_{3,3}(1)$ and $N_{3,2}$ (see Table~1). In \cite[Theorem~6.1]{Hegazi2015}, the algebras $\E_{4,1}$, $\E_{4,2}$, $\E_{4,3}$, $\E_{4,5}$ and $\E_{4,10}$ are associative.

Since, the evolution algebras are not defined by identities and thus do not form a variety of non-associative algebras, such as that of Lie algebras, alternative algebras or Jordan algebras. Therefore, research on these algebras follows different paths \cite{Casado,Casado2016,camacho,Casas2014,Elduque2015,Elduque2016,Hegazi2015,Labra,Tian2006}.


In section 2, we give an example of an algebra which is power-associative but which is not an evolution algebra. We also recall some definitions and known results about nil evolution algebras.

In section 3, we give Wedderburn decomposition of a power-associative evolution algebra. Since the semisimple component is associative, we deduce that an evolution algebra is power-associative if and only if its nil radical is power-associative. We show that the nil power-associative evolution algebras are Jordan algebras. Thus, power-associative evolution algebras are identified with those which are Jordan algebras.

In section 4, we determine power-associative evolution algebras up to dimension six.

\section{Preliminaries}
\subsection{Example of Jordan algebra that is not evolution algebra}

An algebra $\E$ is baric if there is a non trivial homomorphism of algebra $\w : \E \longrightarrow K.$

\medskip
\begin{exe} In \cite[Example~2.2]{Casado2016}, the authors show that zygotic algebra for simple Mendelian inheritance $Z(n,2)$ (case $n=2$) is not an evolution algebra. This algebra is a Jordan algebra. In fact, algebra $Z(n,2)$ is the commutative duplicate of the gametic algebra for simple Mendelian $G(n,2).$
\end{exe}

\medskip
\begin{exe} The  gametic algebra for simple Mendelian inheritance $G(n,2)$ is not an evolution algebra.
\end{exe}

\noindent  \textbf{Proof.} The multiplication table of $G(n,2)$ in the natural basis $B_{1} = \{a_{i}$ ; $1 \leq i \leq n\}$ is given by $a_{i}a_{j} = \frac{1}{2}(a_{i} + a_{j}).$ The linear mapping $\w : a_{i} \longmapsto 1$ is a weight function of $G(n, 2)$ and for all $x, y \in G(n, 2)$, we have $xy = \frac{1}{2}(\w(x)y + \w(y)x)$ (\cite{Buse}). Suppose that $G(n, 2)$ is an evolution algebra in the basis $B_{2} = \{e_{i}$ ; $1 \leq i \leq n\}.$

Since $\w$ is weight function, there is $i_{0} \in \{1, \ldots, n\}$ such that $\w(e_{i_{0}}) \neq 0$. For $j\ne i_0$, $e_{i_0}e_j=0$ leads to $\w(e_{i_0})\w(e_j)=0$, so $\w(e_j)=0$. Then $0=e_{i_0}e_j=\frac12 \w(e_{i_0})e_j$ gives $e_j=0$ impossible. Such a basis $B_2$ does not exist.  \hfill\qed

\subsection{Results about nil evolution algebras}

For an element $a \in \E,$ we define principal powers as follows $a^1=a$ and $a^{k+1}=a^ka$ while that of $\E$ are defined by
\begin{eqnarray*}
\E^{1} = \E,\quad \E^{k + 1} = \E^{k}\E \quad (k\geq 1).
\end{eqnarray*}

\begin{defi}
We will say that algebra $\E$ is:

$i)$ \emph{right nilpotent} if there is a nonzero integer $n$ such that $\E^{n} = 0$ and the minimal such number is called the \textit{index} of right nilpotency~;

$ii)$ \emph{nil} if there is a nonzero integer $n(a)$ such that $a^{n(a)} = 0,$ for all $a \in \E$ and the minimal such number is called the \textit{nil-index} of $\E$.
\end{defi}

\begin{theo}[\cite{Casas2014}, Theorem 2.7] The following statements are equivalent:

$i)$ there is a natural basis in which the matrix of structural constants is of a form

\[\left(
\begin{array}{ccccc}
0 & a_{12} & a_{13} & \cdots & a_{1n} \\
0 & 0 & a_{23} & \cdots & a_{2n} \\
\vdots & \vdots &  \ddots&  & \vdots \\
0 & 0 & \cdots & \ddots& a_{n-1,n} \\
0 & 0 & 0 & \cdots & 0 \\
\end{array}
\right);\]

$ii)$ $\E$ is a right nilpotent algebra ;

$iii)$ $\E$ is a nilalgebra.
\end{theo}

We define the following subspace of $\E$:

$\bullet$ The \textit{annihilator} $\ann^{1}(\E) = \ann(\E) = \{x \in \E : x\E = 0\}$ ;

$\bullet$ $\ann^{i}(\E)$ is defined by $\ann^{i}(\E)/\ann^{i-1}(\E) = \ann\big(\E/\ann^{i-1}(\E)\big).$

$\bullet$ $\U_{i}\oplus \U_{1} = \{x \in \ann^{i}(\E) \mid x \ann^{i-1}(\E) = 0\}.$

In (\cite[Lemme 2.7]{Elduque2015}), the authors show that $\ann(\E) = span\{e_{i} \in B \mid e_{i}^{2} = 0\}.$

In (\cite{Elduque2016}), the authors show that $\ann^{i}(\E) = span\{e \in B \mid e^{2} \in \ann^{i-1}(\E)\}$ and that the basis $B=B_1\cup\cdots \cup B_r$ where $B_i=\{e\in B\mid e^2 \in \ann^{i-1}(\E), e\not\in \ann^{i-1}(\E)\}$. Then, for $\U_i:=span\{B_i\}$, $i=1,\ldots,r$, we have $\U_1\oplus\cdots\oplus\U_i=\ann^i(\E)$ ($i=1,\ldots,r$). They prove then that $\U_{i}\oplus \U_{1}$ is an invariant of nil evolution algebras.

\textit{The type} of a nil evolution algebra $\E$ is the sequence $[n_{1}, n_{2}, \ldots, n_{r}]$ where $r$ and $n_{i}$ are integers defined by $\ann^{r}(\E) = \E$ ; $n_{i} = \dim_{K}(\ann^{i}(\E)) - \dim_{K}(\ann^{i-1}(\E))$ and $n_1+\cdots+n_i=dim_K(\ann^i(\E))$ for all $i \in \{1, \ldots, r\}.$

\section{Power associativity}

\subsection{Some definitions}

\begin{defi} Algebra $\E$ is called:

$i)$ \textit{associative} if for all $x,$ $y,$ $z \in \E,$ $(x, y, z) = 0$ where $(x, y, z) = (xy)z - x(yz)$ is the associator of the elements $x,$ $y,$ $z$ of $\E$ ;

$ii)$ \textit{Jordan} if it is commutative and $(x^{2}, y, x) = 0$, for all $x, y \in \E$ ;

$iii)$ \textit{power-associative} if the subalgebra generated by $x$ is associative. In the other words, for all  $x \in \E$, $x^{i}x^{j} = x^{i + j}$ for all  integers $i, j\geq 1.$
\end{defi}

\medskip\begin{defi}
Algebra $\E$ is said to be \textit{third power-associative} if for all $x \in \E,$ $x^{2}x = xx^{2}.$
\end{defi}

\begin{rem}
Since any commutative algebra is third power-associative, we deduce that any evolution algebra is third power-associative.
\end{rem}

\begin{defi}
Algebra $\E$ is said to be \textit{fourth power-associative} if $x^{2}x^{2} = x^{4}$ for all $x \in E.$
\end{defi}

\begin{theo}[\cite{Albert1}] Let $F$ be a field of characteristic $\neq 2, 3, 5$. Algebra $\E$ is power-associative if and only if $x^{2}x^{2} = x^{4}$, for all $x \in \E.$
\end{theo}

\subsection{Characterization of associative algebras}
If $\E$ is a $n$-dimensional algebra with basis $\{e_{i}$ ; $1 \leq i \leq n\},$ then $\E$ is associative if and only if $(e_{i}e_{j})e_{k} = e_{i}(e_{j}e_{k})$ for all $1 \leq i, j , k  \leq n$ (\cite[Proposition 1]{Jacobson}). We then deduce:

\begin{lem}\label{asso} An evolution algebra with natural basis $B = \{e_{i}$ ; $1 \leq i \leq n\}$ is associative if and only if $e_{i}^{2}e_{j} = e_{i}(e_{i}e_{j}) = 0$ for all $1 \leq i \neq j \leq n.$
\end{lem}

\begin{theo}[\cite{Myung}, Theorem 1]\label{Myung}
A commutative power-associative nilalgebra $A$ of nilindex $3$
and of dimension $4$ over a field $F$ of characteristic $\ne 2$ is associative, and $A^3 = 0$.
\end{theo}

For the evolution algebra, this theorem has the following generalization:

\begin{theo} Let $\E$ be a finite-dimensional evolution algebra. Then $\E$ is an associative nilalgebra if and only if $x^{3} = 0$, for all $x \in \E$. In this case $\E^{3} = 0$.
\end{theo}

\begin{proof}

Let $B = \{e_{i}$ ; $1 \leq i \leq n\}$ be a natural basis of $\E.$

Let's suppose that $\E$ is an associative nilalgebra. Then for all integer $i$ we have $e_i^3=0$. Let's set $x=\sum_{i=1}^nx_ie_i$. We have  $x^3=\sum_{i,j=1}^n x_i^2x_je_i^2e_j=0$ because $e_i^2e_j=0$ since $\E$ is nil and associative.

Conversely the partial linearization of the identity $x^3=0$ give
\begin{eqnarray}\label{Eqb}
2(xy)x + x^{2}y = 0.
\end{eqnarray}
For $i$ distinct from $j$ we have $e_i^2e_j=0$ and $\E$ is associative by Lemma~\ref{asso}.
We have
$\E^{3} = \big<e_{i}^{2}e_{j}$ ; $1 \leq i, j \leq n \big> = 0$.
\end{proof}

\subsection{Characterization of power-associative algebras}
In the following a field $F$ is of characteristic $\neq 2$.

$\E$ is a $n$-dimensional evolution algebra over a field $F$ and with a natural basis $B = \{e_{i}$ ; $1 \leq i \leq n\}$ whose multiplication table is given by \eqref{Eq1}.

\begin{lem}\label{paa}
Algebra $\E$ is a fourth power-associative if and only if the following conditions are satisfied:

$1)$ $e_{i}^{4} = e_{i}^{2}e_{i}^{2}$ for all $1 \leq i \leq n$ ;

$2)$ $2e_{i}^{2}e_{j}^{2} = (e_{i}^{2}e_{j})e_{j} + (e_{j}^{2}e_{i})e_{i}$ for all $1 \leq i < j \leq n$ ;

$3)$ $e_{i}^{3}e_{j} + (e_{i}^{2}e_{j})e_{i} = 0$ for all $ 1 \leq i \neq j \leq n$ ;

$4)$ $(e_{i}^{2}e_{j})e_{k} + (e_{i}^{2}e_{k})e_{j} = 0$ for all $1 \leq i \leq n$ and $1 \leq j < k \leq n$ with $i \neq j$ and $i \neq k.$
\end{lem}

\begin{proof}
Suppose that $\E$ is a fourth power-associative algebra and $x, y, z \in \E.$ We have
\begin{equation} \label{eq2}
(x, x, x^{2}) = 0.
\end{equation}
The partial linearization of (\ref{eq2}) gives:
\begin{equation} \label{eq3}
(y, x, x^{2}) + (x, y, x^{2}) + 2(x, x, xy) = 0
\end{equation}
\begin{equation} \label{eq4}
(y, z, x^{2}) + 2(y, x, xz) + (z, y, x^{2}) + 2(x, y, xz) + 2(z, x, xy) + 2(x, z, xy) + 2(x, x, zy) = 0
\end{equation}
Since the basis vectors $\{e_{i}$ ; $1 \leq i \leq n\}$ check identities (\ref{eq2}), (\ref{eq3}) and (\ref{eq4}), we have

$(\ref{eq2}) \Longrightarrow 1)$ taking $x = e_{i}.$

$(\ref{eq3}) \Longrightarrow 3)$ taking $x = e_{i}$ and $y = e_{j}$ with $i \neq j.$

$(\ref{eq4}) \Longrightarrow 2)$ taking $x = e_{i}$ and $y = z = e_{j}$ with $i \neq j.$

$(\ref{eq4}) \Longrightarrow 4)$ taking $x = e_{i}$, $y = e_{j}$ and $z = e_{k}$ ;  $i, j, k$ are pairwise distinct.


Conversely, suppose that conditions $1),$ $2),$ $3)$ and $4)$ are satisfied. Let $x = \sum_{i=1}^{n} x_{i}e_{i}$ be an element of $\E.$ We have the following equalities :
\begin{eqnarray*}
x^{2} & = & \sum_{i=1}^{n} x_{i}^{2}e_{i}^{2}, \\
x^{3} & = & \sum_{i, j =1}^{n} x_{i}^{2}x_{j}e_{i}^{2}e_{j} =  \sum_{i=1}^{n} x_{i}^{3}e_{i}^{3} + \sum_{i = 1}^{n}\sum\limits_{\substack{j = 1 \\ j \neq i}}^{n} x_{i}^{2}x_{j}e_{i}^{2}e_{j},\\
x^{2}x^{2} & = & \sum_{i, j = 1}^{n} x_{i}^{2}x_{j}^{2}e_{i}^{2}e_{j}^{2}=  \sum_{i = 1}^{n} x_{i}^{4}e_{i}^{2}e_{i}^{2}  + \sum_{i = 1}^{n}\sum\limits_{\substack{j = 1 \\ j \neq i}}^{n}x_{i}^{2}x_{j}^{2}e_{i}^{2}e_{j}^{2}\\
& = & \sum_{i = 1}^{n} x_{i}^{4}e_{i}^{2}e_{i}^{2}  + 2\sum_{1 \leq i < j \leq n}x_{i}^{2}x_{j}^{2}e_{i}^{2}e_{j}^{2},\\
x^{4} & = & \sum_{i, j =1}^{n} x_{i}^{3}x_{j}e_{i}^{3}e_{j} + \sum_{i = 1}^{n}\sum\limits_{\substack{j = 1 \\ j \neq i}}^{n}\sum_{k=1}^{n} x_{i}^{2}x_{j}x_{k}(e_{i}^{2}e_{j})e_{k}\\
& = &  \sum_{i = 1}^{n} x_{i}^{4}e_{i}^{4} + \sum_{i = 1}^{n}\sum\limits_{\substack{j = 1 \\ j \neq i}}^{n} x_{i}^{3}x_{j}\big(e_{i}^{3}e_{j} + (e_{i}^{2}e_{j})e_{i}\big) + \sum_{i=1}^{n}\sum\limits_{\substack{j = 1 \\ j \neq i}}^{n}x_{i}^{2}x_{j}^{2}(e_{i}^{2}e_{j})e_{j}{}\nonumber\\
& & {} + \sum_{i=1}^{n}\sum\limits_{\substack{j = 1 \\ j \neq i}}^{n}\sum\limits_{\substack{k = 1 \\ \substack{k \neq i \\ k\neq j}}}^{n} x_{i}^{2}x_{j}x_{k}(e_{i}^{2}e_{j})e_{k}.
\end{eqnarray*}
We have
\begin{eqnarray*} \sum_{i=1}^{n}\sum\limits_{\substack{j = 1 \\ j \neq i}}^{n} x_{i}^{2}x_{j}^{2}(e_{i}^{2}e_{j})e_{j} & = & \sum_{1 \leq i < j \leq n}x_{i}^{2}x_{j}^{2}\big((e_{i}^{2}e_{j})e_{j} + (e_{j}^{2}e_{i})e_{i}\big),
\end{eqnarray*}
and
\begin{eqnarray*}
\sum_{i=1}^{n}\sum\limits_{\substack{j = 1 \\ j \neq i}}^{n}\sum\limits_{\substack{k = 1 \\ \substack{k \neq i \\ k\neq j}}}^{n} x_{i}^{2}x_{j}x_{k}(e_{i}^{2}e_{j})e_{k} & = & \sum_{i=1}^{n}\sum\limits_{\substack{1 \leq j < k \leq n \\ \substack{j \neq i \\ k \neq i}}} x_{i}^{2}x_{j}x_{k}\big((e_{i}^{2}e_{j})e_{k} + (e_{i}^{2}e_{k})e_{j}\big) =  0.
\end{eqnarray*}
So
\begin{eqnarray*}x^{4} & = & \sum_{i = 1}^{n} x_{i}^{4}e_{i}^{4} + \sum_{1 \leq i < j \leq n}x_{i}^{2}x_{j}^{2}\big((e_{i}^{2}e_{j})e_{j} + (e_{j}^{2}e_{i})e_{i}\big) =  x^{2}x^{2}.
\end{eqnarray*}
\end{proof}

\begin{coro}\label{PAA} Let $\E$ be a nil evolution algebra. Then $\E$ is fourth power-associative algebra if and only if the following conditions are satisfied:

$1)$ $e_{i}^{2}e_{j}^{2} = 0$ for all $1 \leq i \leq j \leq n$ ;

$2)$ $(e_{i}^{2}e_{j})e_{k} = 0$ for all $1 \leq i, j, k \leq n.$

\end{coro}

\begin{proof}
Since $\E$ is a nilalgebra, then for all $1 \leq i,j, k \leq n$, we have $e_i^3=0$ and $a_{jk}a_{kj}=0$ \cite[Proof of theorem~2.2]{Casas2014}. According to $2)$ of Lemma~\ref{paa}, we have
$2e_i^2e_j^2=(e_i^2e_j)e_j+(e_j^2e_i)e_i=a_{ij}e_j^3+a_{ji}e_i^3=0$, hence $1)$.

Let's show that $(e_{i}^{2}e_{j})e_{k} = 0$ for $i, j, k$ two by two distinct and consider $4)$ of Lemma~\ref{paa}, i.e.
\begin{equation} \label{Eq6}
0 = (e_{i}^{2}e_{j})e_{k} + (e_{i}^{2}e_{k})e_{j} = a_{ij}a_{jk}e_{k}^{2} + a_{ik}a_{kj}e_{j}^{2}.
\end{equation}
Since $a_{jk}a_{kj}=0$, we consider two cases:

If $a_{jk} = 0$, then $(e_{i}^{2}e_{j})e_{k} = 0$ and \eqref{Eq6} leads to $(e_{i}^{2}e_{k})e_{j} = 0$;

If $a_{kj} = 0$, then $(e_i^2e_k)e_j=0$ and \eqref{Eq6} gives $(e_i^2e_k)e_j=0$;

thus $4)$ of Lemma~\ref{paa} is equivalent to $(e_i^2e_j)e_k=0$, for all $i, j, k$ pairwise distinct.
\end{proof}

\subsection{Characterization of Jordan algebras}

\begin{lem}\label{Jordan}
$\E$ is Jordan algebra if and only if the following conditions are satisfied:

$1)$ $e_{i}^{2}e_{i}^{2} = e_{i}^{4}$ for all $1 \leq i \leq n$ ;

$2)$ $e_{i}^{3}e_{j} = 0$ for all $1 \leq i \neq j \leq n$ ;

$3)$ $(e_{i}^{2}e_{j})e_{i} = 0$ for all $1 \leq i \neq j \leq n$ ;

$4)$ $e_{i}^{2}e_{j}^{2} = (e_{i}^{2}e_{j})e_{j} = (e_{j}^{2}e_{i})e_{i}$ for all $1 \leq i \neq j \leq n$ ;

$5)$ $(e_{i}^{2}e_{j})e_{k} = 0$ for all $1 \leq i, j, k \leq n$ with $i, j, k$ pairwise distinct.
\end{lem}

\begin{proof}
 Suppose that $\E$ is Jordan algebra and $x, y, z \in \E$. We have
\begin{equation} \label{eq10}
(x^{2}, y, x) = 0.
\end{equation}
By linearizing (\ref{eq10}), we obtain
\begin{equation} \label{eq11}
2(xz, y, x) + (x^{2}, y, z) = 0
\end{equation}
Basis vectors $\{e_{i}$ ; $1 \leq i \leq n\}$ verify identities (\ref{eq10}), (\ref{eq11}).

$(\ref{eq10}) \Longrightarrow 1)$ ; taking $x = y = e_{i}$ ;

$(\ref{eq10}) \Longrightarrow 3)$ ; taking $x = e_{i}$ and $y = e_{j}$ with $i \neq j$ ;

$(\ref{eq11}) \Longrightarrow 2)$ ; taking $x = y = e_{i}$ and $z = e_{j}$  with $i \neq j$ ;

$(\ref{eq11}) \Longrightarrow 4)$ ; taking $x = e_{i}$ and $y = z = e_{j}$ ; we also take $x = e_{j}$ and $y = z = e_{i}$ with $i \neq j$ ;

$(\ref{eq11}) \Longrightarrow 5)$ ; taking $x = e_{i},$ $y = e_{j}$ and $z = e_{k}$  $i, j, k$ pairwise distinct.

\medskip Conversely, suppose that conditions  $1)$, $2)$, $3)$, $4)$ and $5)$ are satisfied. Let $\sum_{i=1}^{n} x_{i}e_{i}$ and $y = \sum_{i=1}^{n} y_{i}e_{i}$ be elements of $\E.$ We have the following equalities:
\begin{align*}
x^{2} & = \sum_{i=1}^{n} x_{i}^{2}e_{i}^{2}, \quad yx  =  \sum_{i=1}^{n} y_{i}x_{i}e_{i}^{2}, \\
x^{2}y & =  \sum_{i,j=1}^{n} x_{i}^{2}y_{j}e_{i}^{2}e_{j} =  \sum_{i=1}^{n} x_{i}^{2}y_{i}e_{i}^{3} + \sum_{i = 1}^{n}\sum\limits_{\substack{j = 1 \\ j \neq i}}^{n}x_{i}^{2}y_{j}e_{i}^{2}e_{j},\\
(x^{2}y)x & =  \sum_{i,j=1}^{n} x_{i}^{2}y_{i}x_{j}e_{i}^{3}e_{j} + \sum_{i = 1}^{n}\sum\limits_{\substack{j = 1 \\ j \neq i}}^{n}\sum_{k = 1}^{n}x_{i}^{2}y_{j}x_{k}(e_{i}^{2}e_{j})e_{k}\\
& =  \sum_{i = 1}^{n} x_{i}^{3}y_{i}e_{i}^{4} +  \sum_{i = 1}^{n}\sum\limits_{\substack{j = 1 \\ j \neq i}}^{n}x_{i}^{2}y_{i}x_{j}e_{i}^{3}e_{j} +
\sum_{i = 1}^{n}\sum\limits_{\substack{j = 1 \\ j \neq i}}^{n}x_{i}^{3}y_{j}(e_{i}^{2}e_{j})e_{i} + \sum_{i = 1}^{n}\sum\limits_{\substack{j = 1 \\ j \neq i}}^{n}x_{i}^{2}y_{j}x_{j}(e_{i}^{2}e_{j})e_{j} {}\nonumber\\
& {} + \sum_{i = 1}^{n}\sum\limits_{\substack{j = 1 \\ j \neq i}}^{n}\sum\limits_{\substack{k = 1 \\ \substack{k \neq i \\ k\neq j}}}^{n}x_{i}^{2}y_{j}x_{k}(e_{i}^{2}e_{j})e_{k}, \\
(x^{2}y)x & =  \sum_{i=1}^{n} x_{i}^{3}y_{i}e_{i}^{4} + \sum_{i = 1}^{n}\sum\limits_{\substack{j = 1 \\ j \neq i}}^{n}x_{i}^{2}y_{j}x_{j}(e_{i}^{2}e_{j})e_{j},\\
x^{2}(yx) & =  \sum_{i, j = 1}^{n} x_{i}^{2}y_{j}x_{j}e_{i}^{2}e_{j}^{2} =  \sum_{i=1}^{n} x_{i}^{3}y_{i}e_{i}^{2}e_{i}^{2} + \sum_{i = 1}^{n}\sum\limits_{\substack{j = 1 \\ j \neq i}}^{n}x_{i}^{2}y_{j}x_{j}e_{i}^{2}e_{j}^{2}\\
& = \sum_{i=1}^{n} x_{i}^{3}y_{i}e_{i}^{4} + \sum_{i = 1}^{n}\sum\limits_{\substack{j = 1 \\ j \neq i}}^{n}x_{i}^{2}y_{j}x_{j}(e_{i}^{2}e_{j})e_{j},\\
x^{2}(yx)& =  (x^{2}y)x.
\end{align*}
\end{proof}

\begin{coro}\label{Jordan2}  Let $\E$ be a nil evolution algebra. Then $\E$ is Jordan algebra if and only if the following conditions are satisfied:

$1)$ $e_{i}^{2}e_{j}^{2} = 0$ for all $1 \leq i \leq j \leq n$ ;

$2)$ $(e_{i}^{2}e_{j})e_{k} = 0$ for all $1 \leq i, j, k \leq n.$
\end{coro}

\begin{proof} Since $2)$ is exactly $5)$ of lemma~\ref{Jordan}, it remains to show $1)$.
Since $\E$ is nilalgebra, then for all $1 \leq i \leq n$, we have  $e_i^3=0$. According to $4)$ of Lemma~\ref{Jordan}, we have
$e_i^2e_j^2=(e_i^2e_j)e_j=a_{ij}e_j^3=0$, hence $1)$.
\end{proof}

\begin{prop}\label{Jordan3} Let $\E$ be a nil finite-dimensional evolution algebra. The following statements are equivalent:

$1)$ $\E$ is fourth power-associative algebra.

$2)$ $\E$ is Jordan algebra.
\end{prop}

\begin{proof} This proposition is a consequence of Corollary~\ref{PAA} and Corollary~\ref{Jordan2}.
\end{proof}

\subsection{Wedderburn decomposition}

Any finite-dimensional commutative power-associative algebra which is not nilalgebra contains at least one nonzero idempotent (\cite{Schafer}, Proposition 3.3 P 33).

Moreover, in (\cite{Albert1}) the Author shows that any commutative power-associative algebra $\E,$ which has a nonzero idempotent, admits the following Peirce decomposition: $\E = \E_{1}(e)\oplus \E_{0}(e)\oplus \E_{\frac{1}{2}}(e)$ where $\E_{\lambda}(e) = \{x \in \E$ ; $ex = \lambda x\}.$

\begin{defi} (\cite{Casado2016}, D{e}finition 2.4)\label{extension}
An evolution subalgebra of an evolution algebra $\E$ is a subalgebra $\E' \subseteq \E$ such that $\E'$ is an evolution algebra i.e. $\E'$ has a natural basis. We say that $\E'$ has the \textit{extension property} if there exists a natural basis $B'$ of $\E'$ which can be extended to a natural basis $B$ of $\E.$
\end{defi}

$\E$ is a $n$-dimensional evolution algebra over
the field $F$ of characteristic $\neq 2$ and with a natural basis $B = \{e_{i}$ ; $1 \leq i \leq n\}$ whose multiplication table is given by \eqref{Eq1}.

\begin{lem}
 If $\E$ is fourth power-associative, then for all $i \neq j,$ $e_{i}^{3}e_{j} = (e_{i}^{2}e_{j})e_{i} = 0.$
\end{lem}

\begin{proof} Suppose that there exists $i_{0} \neq j_{0} \in \{1, \ldots, n\}$ such that $e_{i_{0}}^{3}e_{j_{0}} \neq 0$. Then $a_{i_{0}i_{0}}a_{i_{0}j_{0}}e_{j_{0}}^{2} \neq 0.$
Since $\E$ is power-associative, $0 = e_{i_{0}}^{3}e_{j_{0}} + (e_{i_{0}}^{2}e_{j_{0}})e_{i_{0}} =  a_{i_{0}i_{0}}a_{i_{0}j_{0}}e_{j_{0}}^{2} + a_{i_{0}j_{0}}a_{j_{0}i_{0}}e_{i_{0}}^{2}.$ And since $a_{i_{0}j_{0}} \neq 0$, then
\begin{eqnarray}\label{eq12}
a_{i_{0}i_{0}}e_{j_{0}}^{2} + a_{j_{0}i_{0}}e_{i_{0}}^{2} = 0.
\end{eqnarray}
By multiplying (\ref{eq12}) by $e_{i_{0}},$ we get $0 = a_{i_{0}i_{0}}e_{j_{0}}^{2}e_{i_{0}} + a_{j_{0}i_{0}}e_{i_{0}}^{2}e_{i_{0}} = a_{i_{0}i_{0}}a_{j_{0}i_{0}}e_{i_{0}}^{2} + a_{j_{0}i_{0}}a_{i_{0}i_{0}}e_{i_{0}}^{2} = 2a_{i_{0}i_{0}}a_{j_{0}i_{0}}e_{i_{0}}^{2}$ then $a_{j_{0}i_{0}} = 0.$\\
$(\ref{eq12})$ leads to $a_{i_{0}i_{0}}e_{j_{0}}^{2} = 0$: impossible. We deduce that for all $i \neq j$, $e_{i}^{3}e_{j} = 0$ and $(e_{i}^{2}e_{j})e_{i} = 0.$
\end{proof}

\begin{lem}
Let $\E$ be a fourth power-associative evolution algebra and not nilalgebra.
Then $\E$ admits a nonzero idempotent $e$ such that the evolution subalgebra of $\E$ generated by $e$ has extension property.
\end{lem}

\begin{proof} Since $\E$ is not nilalgebra, there exists $i_{0} \in \{1, \ldots, n\}$ such that $e_{i_{0}}^{3} \neq 0$. Without losing the generality, we can assume that $e_{1}^{3} \neq 0$. Thus $e=a_{11}^{-2}e_1^2$ is an idempotent. Indeed, $e^{2} = a_{11}^{-4}e_{1}^{2}e_{1}^{2} = a_{11}^{-4}e_{1}^{4} = a_{11}^{-2}e_{1}^{2}=e$.
Let's show that the family $\{e, e_{2}, e_{3}, \ldots, e_{n}\}$ is a natural basis of $\E.$ Let $\alpha_{1}, \ldots, \alpha_{n}$ be scalars such that
\begin{equation} \label{eq13}
0 = \alpha_{1}e + \alpha_{2}e_{2} + \cdots + \alpha_{n}e_{n}.
\end{equation}
By multiplying (\ref{eq13}) by $e,$ we get
\begin{equation} \label{eq14}
0 = \alpha_{1}e^{2} + a_{11}^{-2}(\alpha_{2}e_{1}^{2}e_{2} + \cdots + \alpha_{n}e_{1}^{2}e_{n}) = \alpha_{1}e + a_{11}^{-2}(\alpha_{2}e_{1}^{2}e_{2} + \cdots + \alpha_{n}e_{1}^{2}e_{n})
\end{equation}
For $i \neq 1,$  $0 = e_{1}^{3}e_{i} = a_{11}e_{1}^{2}e_{i}$ leads to $e_{1}^{2}e_{i} = 0$ because $a_{11} \neq 0$. Thus $(\ref{eq14})$ leads to $\alpha_{1}e = 0$, then $\alpha_{1} = 0$ and $(\ref{eq13})$ leads to $\alpha_{2} = \alpha_{3} = \cdots = \alpha_{n} = 0$ because $\{e_{2}, e_{3}, \ldots, e_{n}\}$ is linearly independent. It has been shown in passing that $ee_i=0$ for $i=2,\ldots,n$. So the family $\{e, e_{2}, e_{3}, \ldots, e_{n}\}$ is a natural basis of $\E.$
\end{proof}

\begin{theo}[of Wedderburn]\label{Wedder} Let $\E$ be a not nil power-associative evolution algebra. Then $\E$ admits $s$ nonzero pairwise orthogonal idempotents $u_{1}, u_{2}, \ldots, u_{s}$, such that
\begin{equation}\label{wed}
\E = Ku_{1}\oplus Ku_{2}\oplus \cdots \oplus Ku_{s} \oplus N,
\end{equation}
direct sum of algebras, where $s \geq 1$ is an integer and $N$ is either zero or a nil power-associative evolution algebra.

Furthermore $\E_{ss} = Ku_{1}\oplus Ku_{2}\oplus \cdots \oplus Ku_{s}$ is the semisimple component of $\E$ and $N = Rad(\E)$ is the nil radical of $\E.$
\end{theo}

\begin{proof} Suppose
that $\E$ is a not nilalgebra. Then $\E$ admits nonzero idempotent $u_{1}$ such that the evolution subalgebra $Ku_1$ generated by $u_{1}$ has extension property and we denote $\E = \E_{1}(u_{1})\oplus \E_{0}(u_{1})\oplus \E_{\frac{1}{2}}(u_{1})$ the Peirce decomposition of $\E$ relative
to $u_{1}$. Then $\{u_{1}, e_{2}, e_{3}, \ldots, e_{n}\}$ is a natural basis of $\E$.

Let's $v = \alpha_{1}u_{1} + \alpha_{2}e_{2} + \alpha_{3}e_{3} + \cdots + \alpha_{n}e_{n} \in \E_{\frac{1}{2}}(u_{1})$.
We have $u_{1}v = \alpha_{1}u_{1}^{2} = \alpha_{1}u_{1} = \frac{1}{2}v = \frac{1}{2}(\alpha_{1}u_{1} + \alpha_{2}e_{2} + \alpha_{3}e_{3} + \cdots + \alpha_{n}e_{n})$. Then $\alpha_{1} = \frac{1}{2}\alpha_{1}$ and $\alpha_{i} = 0$, for $i \neq 1$. We also have $\alpha_{1} = 0$ and $v = 0,$ i.e. $\E_{\frac{1}{2}}(u_{1}) = 0.$

Now let's $v = \alpha_{1}u_{1} + \alpha_{2}e_{2} + \alpha_{3}e_{3} + \cdots + \alpha_{n}e_{n} \in \E_{1}(u_{1}).$
We have $u_{1}v =  \alpha_{1}u_{1} = v = \alpha_{1}u_{1} + \alpha_{2}e_{2} + \alpha_{3}e_{3} + \cdots + \alpha_{n}e_{n}$ leads to $\alpha_{i} = 0$ for $i \neq 1$. Thus, $\E_{1}(u_{1}) = Ku_{1},$ so $\E = Ku_{1}\oplus \E_{0}(u_{1})$ direct sum of algebras.

If $\E_{0}(u_{1})$ is a nilalgebra, it is finished and $s = 1$. Otherwise we repeat the reasoning with the evolution subalgebra  $\E_{0}(u_{1})$ of dimension $n - 1.$ And so on, we get result because $\dim_{K}(\E)$ is finite.
\end{proof}

\begin{theo}\label{Jordan4} The following statements are equivalent:

$1)$ $\E$ is power-associative algebra.

$2)$ $\E$ is Jordan algebra.
\end{theo}

\begin{proof} $2) \Longrightarrow 1)$ because any Jordan algebra is power-associative.

$1) \Longrightarrow 2)$: Suppose that $\E$ is power-associative. Let $\E = \E_{ss}\oplus N$ be the Wedderburn decomposition of $\E$ where $\E_{ss}$ is the semisimple
component of $\E$ and $N$ is the nil radical of $\E.$

Since $\E_{ss}$ is associative and that \eqref{wed} is a direct somme of algebras, then $\E$ is a Jordan algebra if and only if $N$ is a Jordan algebra. And since $N$ is a nil fourth power-associative evolution algebra,
then it is Jordan algebra by Proposition~\ref{Jordan3}
\end{proof}

\section{Classification}
According to Theorem~\ref{Wedder} and Theorem~\ref{Jordan4}, classification problem of power-associative evolution algebras is reduced to that of  nil power-associative evolution algebras.

\medskip
\begin{defi}\label{indecom} An algebra $\E$ is \textit{decomposable} if there are nonzero ideals $\I$ and $\J$ such that $\E=\I\oplus\J.$ Otherwise, it is \textit{indecomposable}.
\end{defi}

\begin{lem}[\cite{Elduque2016}, Corollary~2.6] Let $\E$ be a finite-dimensional evolution algebra such that $\dim_F(\ann(\E))\geq \frac12\dim_F(\E)\geq 1$. Then $\E$ is decomposable.
\end{lem}

Moreover for classification, it is sufficient to determine the indecomposable evolution algebras.

We are interested in the classification of indecomposable evolution algebras up to dimension 6. Then $\dim_F(\ann(\E))=1$ or $2$.

\subsection{Nil indecomposable associative evolution algebra}

Suppose that $N$ is an associative algebra. We have $e_{i}^{2}e_{j} = 0,$ for all $i, j = 1, \ldots, n.$ Then for all $ i \in \{1, \ldots, n\}$ such that $e_{i}^{2} \neq 0,$ we have $e_{i}^{2} \in \ann(N)$.

1) $\dim(\ann(N)) = 1.$

Without losing the generality, we can suppose that $\ann(N) = Ke_{n}$.
Thus, for all $i \neq n$, we have $e_{i}^{2} = \alpha_{i}e_{n}$ with $\alpha_{i} \neq 0$.
We set $v_{i} = e_{i}$ and $v_{n} = \alpha_{1}e_{n}$ with $1\leq i \leq n - 1.$ We have:\\
$v_{1}^{2} =  v_{n}$, $v_{i}^{2}  =  \alpha_{i}\alpha_{1}^{-1}v_{n} = \beta_{i}v_{n}$, $v_{n}^{2}  =  0$ with $\beta_{i} \neq 0$ and $2 \leq i \leq n - 1$.

2) $\dim(\ann(N)) = 2.$

Without losing the generality, we can suppose that $\ann(N) = Ke_{n - 1}\oplus Ke_{n}$.
Thus, for all $i \neq n - 1, n$, we have $e_{i}^{2} = \al_{i, n - 1}e_{n - 1} + \al_{i, n}e_{n}$ with $(\al_{i, n - 1}, \al_{in}) \neq 0$.
We can suppose that $N^{2} = Ke_{1}^{2}\oplus Ke_{n - 2}^{2}$, otherwise we swap the vectors $e_{1}, e_{2}, \ldots, e_{n - 2}$.
We set $v_{i} = e_{i}$ ($1 \leq i \leq n$), $v_{n - 1} = v_{1}^{2}$, $v_{n} = v_{n - 2}^{2}$,  $v_{i}^{2} = \alpha_{i, n - 1}v_{n - 1} +\alpha_{i, n }v_{n }$. Also, there is $i_{0} \in \{2, \ldots, n - 3\}$ such that $\alpha_{i_{0}, n - 1}\alpha_{i_{0}n} \neq 0,$ otherwise $N$ would be decomposable.

\subsubsection{One-dimensional classification}
\begin{theo} The only $1$-dimensional nil evolution algebra, indecomposable and associative, is $N_{1,1}$: $e_{1}^{2} = 0$ of type $[1]$.
\end{theo}

\subsubsection{Two-dimensional classification}
\begin{theo} The only $2$-dimensional nil evolution algebra indecomposable and associative is $N_{2,2}$: $e_{1}^{2} = e_{2}$, $e_{2}^{2} = 0$, of type $[1, 1]$.
\end{theo}

\begin{proof} $\dim(N) = 2$ leads to $\dim(\ann(N)) = 1$, otherwise $N$ would be decomposable.
So, $N$ is isomorphic to $N_{2,2}$: $e_{1}^{2} = e_{2}$ and $e_{2}^{2} = 0.$
\end{proof}

\subsubsection{Three-dimensional classification}
\begin{theo} The only $3$-dimensional nil evolution algebra indecomposable and associative is $N_{3,3}(\alpha)$: $e_{1}^{2} = e_{3}$, $e_{2}^{2} = \alpha e_{3}$, $e_{3}^{2} = 0$, of type $[1, 2]$ where $\alpha \in F^{*}$.
\end{theo}

\begin{proof} $\dim(N) = 3$ leads to $\dim(\ann(N)) < \frac{1}{2}\dim(N) = 1.5$, so $\dim(\ann(N)) = 1$. Thus, $N$ is isomorphic to $N_{3,3}(\alpha)$: $e_{1}^{2} = e_{3},$ $e_{2}^{2} = \alpha e_{3}$ and $e_{3}^{2} = 0$ with $\alpha \in F^{*}.$
\end{proof}

\subsubsection{Four-dimensional classification}
\begin{theo} The only $4$-dimensional nil evolution algebra indecomposable and associative is $N_{4,5}(\alpha, \beta)$: $e_{1}^{2} = e_{4}$, $e_{2}^{2} = \alpha e_{4}$, $e_{3}^{2} = \beta e_{4}$, $e_{4}^{2} = 0$, of type $[1, 3]$, where $\alpha, \beta \in F^{*}$.
\end{theo}

\begin{proof} $\dim(N) = 4$ leads to $\dim(\ann(N)) < \frac{1}{2}\dim(N) = 2$, so $\dim(\ann(N)) = 1.$
Thus, $N$ is isomorphic to $N_{4,5}(\alpha, \beta)$: $e_{1}^{2} = e_{4},$ $e_{2}^{2} = \alpha e_{4},$ $e_{3}^{2} = \beta e_{4},$ $e_{4}^{2} = 0$ where $\alpha, \beta \in F^{*}.$
\end{proof}

\subsubsection{Five-dimensional classification}
\begin{theo} Let $N$ be a nil indecomposable associative evolution algebra of dimension five. Then $N$ is isomorphic to one and only one of the algebras  $N_{5,8}(\al,\be,\g)$ and $N_{5,9}(\al,\be)$ in Table~2.
\end{theo}

\begin{proof}$\dim(N) = 5$ leads to $\dim(\ann(N)) < \frac{1}{2}\dim(N) = 2.5$, so $\dim(\ann(N)) = 1, 2$.

1) $\dim(\ann(N)) = 1$ ; $N$ is isomorphic to $N_{5,8}(\alpha, \beta, \gamma)$: $e_{1}^{2} = e_{5},$ $e_{2}^{2} = \alpha e_{5},$ $e_{3}^{2} = \beta e_{5},$ $e_{4}^{2} = \gamma e_{5},$ $e_{5}^{2} = 0$ with $\alpha, \beta, \gamma \in F^{*}.$

2) $\dim(\ann(N)) = 2$ ; $N$ is isomorphic to $N_{5,9}(\alpha, \beta)$: $e_{1}^{2} = e_{4},$ $e_{2}^{2} = \alpha e_{4} +  \beta e_{5},$ $e_{3}^{2} = e_{5},$ $e_{4}^{2} = e_{5}^{2} = 0$ with $\alpha, \beta \in F^{*}.$
\end{proof}

\subsubsection{Six-dimensional classification}

\begin{theo} Let $N$ be a nil indecomposable associative evolution algebra of dimension six. Then $N$ is isomorphic to one and only one of the algebras $N_{6,16}(\al,\be,\g,\delta)$, $N_{6,17}(\al,\be,\g)$ and $N_{6,18}(\al,\be,\g,\delta)$ in Table~3.
\end{theo}

\begin{proof}$\dim(N) = 6$ leads to $\dim(\ann(N)) < \frac{1}{2}\dim(N) = 3$ so $\dim(\ann(N)) = 1, 2.$

1) $\dim(\ann(N)) = 1$ ; $N$ is isomorphic to $N_{6,16}(\alpha, \beta, \gamma,\delta)$: $e_{1}^{2} = e_{6},$ $e_{2}^{2} = \alpha e_{6},$ $e_{3}^{2} = \beta e_{6},$ $e_{4}^{2} = \gamma e_{6},$ $e_{5}^{2} = \delta e_{6},$ $e_{6}^{2} = 0$ with $\alpha, \beta, \gamma, \delta \in F^{*}.$

2) $\dim(\ann(N)) = 2$ ; the multiplication table of $N$ is of the form: $e_{1}^{2} = e_{5},$ $e_{2}^{2} = \alpha e_{5} +  \beta e_{6},$ $e_{3}^{2} = \gamma e_{5} + \delta e_{6},$ $e_{4}^{2} = e_{6},$ $e_{5}^{2} = e_{6}^{2} = 0$ with $\alpha\beta \neq 0$ and $(\gamma, \delta) \neq 0.$
\begin{itemize}
  \item[$a)$]  If $\al\de-\be\g\neq 0$ and $\g\de\ne0$, then the algebra is isomorphic to $N_{6,18}(\al,\be,\g,\de)$.
  \item[$b)$] $\g\de=0$ or $\al\de-\g\be=0$. We consider three cases :
\begin{description}
 \item[$i)$] If $\g=0$ ($\de\ne0$), then $N$ is isomorphic to $N_{6,17}(\al,\be,\de)$ : $e_{1}^{2} = e_{5}$, $e_{2}^{2} = \alpha e_{5} +  \beta e_{6}$, $e_{3}^{2} = \de e_6$, $e_{4}^{2} = e_{6}$, $e_{5}^{2} = e_{6}^{2} = 0$.
\item[$ii)$] If $\de=0$ ($\g\ne0$), by swapping $e_5$ and $e_6$ we find again $N_{6,17}(\be,\al,\g)$.
\item[$iii)$] $\al\de-\be\g=0$ and $\g\de\ne0$. We have $\de=\al^{-1}\be\g$ and by setting  $w_1=e_1$, $w_2=e_4$, $w_3=e_3$, $w_4=e_2$, $w_5=e_5$ and $w_6=\al e_5+\be e_6$, we get algebra $N_{6,17}(-\al\be^{-1},\be^{-1},\al^{-1}\g)$.
\end{description}
\end{itemize}
\end{proof}

\subsection{Nil indecomposable evolution algebra which is not associative}
Suppose that $N$ is not associative. There are $i_{0}, j_{0} \in \{1, \ldots, n\}$ distincts such that $e_{i_{0}}^{2}e_{j_{0}} \neq 0$. We have $0 \neq e_{i_{0}}^{2}e_{j_{0}} = a_{i_{0}j_{0}}e_{j_{0}}^{2}$ leads to $a_{i_{0}j_{0}} \neq 0$, $e_{i_{0}}^{2} \neq 0$ and $e_{j_{0}}^{2} \neq 0$.

Since $N$ is power-associative, we have $0 = (e_{i_{0}}^{2}e_{j_{0}})e_{k} = a_{i_{0}j_{0}}e_{j_{0}}^{2}e_{k}$ then $e_{j_{0}}^{2}e_{k} = 0$, for all integers $k\ne j_0$ and since $e_{j_0}^3=0$, then $e_{j_{0}}^{2}e_{k} = 0$, for all integers $k$. We deduce that $e_{j_{0}}^{2} \in \ann(N)$.

Since $e_{i_0}^2\not\in \ann(N)$, necessarily $a_{ii_0}=0$ for all $1\leq i\leq n$ if not, assuming that $a_{i_1i_0}\ne 0$ for some $i_1$,  we have , for all $1 \leq k \leq n$, $0 = (e_{i_{1}}^{2}e_{i_{0}})e_{k} = a_{i_{1}i_{0}}e_{i_{0}}^{2}e_{k}$, so $e_{i_{0}}^{2}e_{k} = 0$ i.e. $e_{i_{0}}^{2} \in \ann(N)$ : contradiction.


Suppose that for all $i \neq i_0, j_{0}$, $e_{i_{0}}^{2}e_{i} = 0$. Since
\begin{equation*}
e_{i_{0}}^{2} = a_{i_{0}j_{0}}e_{j_{0}} + \sum_{i \neq i_{0}, j_{0}}a_{i_{0}i}e_{i},
\end{equation*}
We have
\begin{equation*}
0=e_{i_{0}}^{2}e_{i_{0}}^{2} = a_{i_{0}j_{0}}^{2}e_{j_{0}}^{2} + \sum_{i \neq i_{0}, j_{0}}a_{i_{0}i}^{2}e_{i}^{2} = a_{i_{0}j_{0}}^{2}e_{j_{0}}^{2} \textrm{ because } a_{i_{0}i}^{2}e_{i}^{2} = a_{i_{0}i}e_{i_{0}}^{2}e_{i} = 0.
\end{equation*}
This is impossible because $a_{i_{0}j_{0}}e_{j_{0}}^{2} \neq 0.$
Then there is $j_{1}$ distinct from $i_{0}$ and $j_{0}$ such that $e_{i_{0}}^{2}e_{j_{1}} \neq 0$. The argument at the beginning of the paragraph tells us that $e_{j_{1}}^{2} \in \ann(N).$
Since $N$ is nil, then $\ann(N) \neq 0$, $\dim(N) > 3$ and without losing generality, we will set $i_{0} = 1,$ $j_{0} = 2$ and $j_{1} = 3.$

\begin{theo}[\cite{Myung}, Theorem~2]
Let $A$ be a commutative power-associative nilalgebra of nil-index $4$ and dimension $4$ over a field $F$ of characteristic $\neq$ 2, then $A^{4} = 0$ and there is $y \not\in A^{2}$ such that $yA^{2} = 0$.
\end{theo}
For the evolution algebras, this theorem has the following generalization:
\begin{theo}
Let $\E$ be a finite-dimensional nil power-associative evolution algebra which is not associative. Then $\E$ is of nil-index $4$, $\E^4=0$ and there is $y\not\in \E^2$ such that $y\E^2=0$.
\end{theo}

\begin{proof}
Let's suppose that $\E$ is a power-associative evolution nilalgebra which is not associative and $x=\sum_{i=1}^nx_ie_i$. We have $x^3=\sum_{i,j=1}^nx_i^2x_je_i^2e_j$ and $x^4=\sum_{i,j,k=1}x_i^2x_jx_k(e_i^2e_j)e_k=0$ by Corollary~\ref{PAA}. Since $\E$ is not associative, there are integers $i_0\ne j_0$ such that $e_{i_0}^2e_{j_0}\ne 0$. In this case $e_{j_0}^2e_{i_0}=0$ and then the element $a=e_{i_0}+e_{j_0}$ checks  $a^3=e_{i_0}^2e_{j_0}+e_{j_0}^2e_{i_0}=e_{i_0}^2e_{j_0}\ne 0$. Thus the nilindex of $\E$ is $4$.
We have $\E^{3} = \big<e_{i}^{2}e_{j}$ ; $1 \leq i, j \leq n\big> \neq 0$ because $\E$ is not associative and $\E^{4} = \big<(e_{i}^{2}e_{j})e_{k}$ ; $1 \leq i, j, k \leq n\big> = 0$ by Corollary~\ref{PAA}. Since $e_{i_0}^2\not\in \ann{\E}$ necessarily $a_{ki_0}=0$ and $e_{k}^{2}e_{i_{0}} = a_{ki_{0}}e_{i_{0}}^{2} = 0$ for all integer $k$, hence $e_{i_{0}}\E^{2} = 0$. In fact $e_{i_0}\not\in \E^2$ otherwise we would have $e_{i_{0}}^{2} \in \E^{[3]} = \E^{2}\E^{2} =  \big<e_{i}^{2}e_{j}^{2}$ ; $1 \leq i, j \leq n\big> = 0$: absurd because $e_{i_{0}}^{2} \neq 0$, hence the theorem
\end{proof}

\subsubsection{Four-dimensional classification}
\begin{theo} The only $4$-dimensional nil indecomposable power-associative which is not associative evolution algebra is $N_{4,6}$ in Table~1.
\end{theo}

\begin{proof} Since $\dim(N) = 4$, we have $\dim(\ann(N)) = 1.$ From the above
\begin{eqnarray*}
e_{1}^{2} & = &  a_{12}e_{2} + a_{13}e_{3} + a_{14}e_{4},
e_{2}^{2} =  a_{24}e_{4}, e_{3}^{2}  =   a_{34}e_{4},
e_{4}^{2}  =  0 \textrm{ with }a_{12}a_{13}a_{24}a_{34} \neq 0.
\end{eqnarray*}
Let's set $v_{2} = a_{12}e_{2} + a_{14}e_{4},$  $v_{3} = a_{13}e_{3}$. We have $e_{1}^{2} = v_{2} + v_{3}$ and $0 = e_{1}^{2}e_{1}^{2} = v_{2}^{2} + v_{3}^{2}$ leads to $v_{3}^{2} = -v_{2}^{2}$.
Then we set $v_{4} = v_{2}^{2} = a_{12}^{2}a_{24}e_{4}$ ; so $N$ is isomorphic to
\begin{eqnarray*}
\textrm{$N_{4,6}$ : }e_{1}^{2} & = &  e_{2} + e_{3},\quad
e_{2}^{2}  =  e_{4},\quad
e_{3}^{2}  =  -e_{4},\quad
e_{4}^{2}  =   0
\end{eqnarray*}
of type $[1, 2, 1]$.
\end{proof}

\subsubsection{Five-dimensional classification}
\begin{theo}Let $N$ be a nil indecomposable power-associative which is not associative evolution algebra, of dimension five. Then $N$ is isomorphic to one and only one of the algebras $N_{5,10}(\al)$, $N_{5,11}(\al)$ and $N_{5,12}(\al,\be)$ in Table~2.
\end{theo}

\begin{proof}
We have $\dim(N) = 5$ leads to $\dim(\ann(N)) < \frac{1}{2}\dim(N) = 2.5$, so $\dim(\ann(N)) = 1, 2$.

1) $\dim(\ann(N)) = 1.$
\begin{eqnarray*}
e_{1}^{2} & = &  a_{12}e_{2} + a_{13}e_{3} + a_{14}e_{4} + a_{15}e_{5},
e_{2}^{2} =  a_{25}e_{5}, e_{3}^{2}  =   a_{35}e_{5}\\
e_{4}^{2} & = &  a_{42}e_{2} + a_{43}e_{3} + a_{45}e_{5},
e_{5}^{2}  =  0 \textrm{ with }a_{12}a_{13}a_{25}a_{35} \neq 0 \textrm{ and }e_{4}^{2} \neq 0.
\end{eqnarray*}
Since $0 = (e_{1}^{2}e_{4})e_{j} = a_{14}a_{4j}e_{j}^{2}$, we have $a_{14}a_{4j} = 0$ (with $j = 2, 3$). we then distinguish two cases.

1.1) $a_{14} = 0$ i.e. $e_{1}^{2} = a_{12}e_{2} + a_{13}e_{3} + a_{15}e_{5}$ with $a_{12}a_{13} \neq 0$.
Let's set $v_{2} = a_{12}e_{2} + a_{15}e_{5}$ and $v_{3} = a_{13}e_{3}$.  We have $e_{1}^{2} = v_{2} + v_{3}$ and $0 = e_{1}^{2}e_{1}^{2} = v_{2}^{2} + v_{3}^{2}$, so $v_{3}^{2} = -v_{2}^{2}$.
We set $v_{5} = v_{2}^{2} = a_{12}^{2}a_{25}e_{5}$, so $v_{3}^{2} = -v_{5}$ and there are scalars $\alpha_{42}$, $\alpha_{43}$ and $\alpha_{45}$ not all zero such that $e_{4}^{2} = \alpha_{42}v_{2} + \alpha_{43}v_{3} + \alpha_{45}v_{5}$.
We have $0 = e_{1}^{2}e_{4}^{2} = \alpha_{42}v_{2}^{2} + \alpha_{43}v_{3}^{2} = (\alpha_{42} - \alpha_{43})v_{5}$, so $\alpha_{42} = \alpha_{43}$ ; this relation ensures the power associativity of $N.$
The family $\{e_{1}, v_{2}, v_{3}, e_{4}, v_{5}\}$ is a natural basis of $N$ and its multiplication table is defined by:
\begin{eqnarray*}
e_{1}^{2} & = &  v_{2} + v_{3}, v_{2}^{2}  =   v_{5}, v_{3}^{2}  =   -v_{5},
e_{4}^{2}  =  \alpha_{42}(v_{2} + v_{3}) + \alpha_{45}v_{5}, v_{5}^{2}  =  0 \textrm{ with }(\alpha_{42}, \alpha_{45}) \neq 0.
\end{eqnarray*}
We have $\ann(N) = \big<v_{5}\big>$, $v_{2}, v_{3} \in \ann^{2}(N)$, $e_{1} \in \ann^{3}(N)$ and since $e_{4}^{2} \in \big<v_{2}, v_{3}, v_{5}\big>$ then the possible types of $N$ are $[1, 3, 1]$ or $[1, 2, 2].$

1.1.1) The type of $N$ is $[1, 3, 1]$, then $\alpha_{42} = 0$.
We have\\ $\ann(N) = \big<v_{5}\big>$, $\ann^{2}(N) = \big<v_{2}, v_{3}, e_{4}, v_{5}\big>$ and $\ann^{3}(N) = N$.
So $\U_{3}\oplus \U_{1} = \{x \in \ann^{3}(N)$ ; $x[\ann^{2}(N)] = 0\} = \big<e_{1}, v_{5}\big>$ and $(\U_{3}\oplus \U_{1})^{2} = \big<v_{2} + v_{3}\big>$. $N$ is isomorphic to
$N_{5,10}(\alpha):  e_{1}^{2}  =  e_{2} + e_{3}, e_{2}^{2} =  e_{5}, e_{3}^{2} = -e_{5}, e_{4}^{2} =  \alpha e_{5}, e_{5}^{2}  =  0$ with $\alpha \neq 0$.

1.1.2) The type of  $N$ is $[1, 2, 2]$, then $\alpha_{42} \neq 0.$

We have $\ann(N) = \big<v_{5}\big>$, $\ann^{2}(N) = \big<v_{2}, v_{3}, v_{5}\big>$ and $\ann^{3}(N) = N.$
So $\U_{3}\oplus \U_{1} = \big<e_{1}, e_{4}, v_{5}\big>$ and $(\U_{3}\oplus \U_{1})^{2} = \big<v_{2} + v_{3}, e_{4}^{2}\big>.$ Thus,
$\dim\big(\U_{3}\oplus \U_{1})^{2}\big) = 1, 2.$

$a)$ $\dim\big((\U_{3}\oplus \U_{1})^{2}\big) = 1$ implies $\alpha_{45} = 0$. $N$ is isomorphic to
 $N_{5,11}(\alpha): e_{1}^{2} =  e_{2} + e_{3}, e_{2}^{2} =  e_{5}, e_{3}^{2} =  -e_{5}, e_{4}^{2} =  \alpha(e_{2} + e_{3}),
e_{5}^{2} =  0$ with $\alpha \neq 0$.

$b)$ $\dim\big((\U_{3}\oplus \U_{1})^{2}\big) = 2$ implies $\alpha_{45} \neq 0$.
$N$ is isomorphic to
$N_{5,12}(\alpha, \beta): e_{1}^{2} = e_{2} + e_{3}, e_{2}^{2} = e_{5}, e_{3}^{2} = -e_{5},
e_{4}^{2} = \alpha(e_{2} + e_{3}) + \beta e_{5}, e_{5}^{2} = 0$ with $\alpha\beta \neq 0$.

1.2) $a_{14} \neq 0$ i.e. $e_{1}^{2} =  a_{12}e_{2} + a_{13}e_{3} + a_{14}e_{4} + a_{15}e_{5}$ and $e_{4}^{2} = a_{45}e_{5}$ with $a_{12}a_{13}a_{14}a_{45} \neq 0$.
Let's set $v_{2} = a_{12}e_{2} + a_{15}e_{5},$ $v_{3} = a_{13}e_{3}$ and $v_{4} = a_{14}e_{4}$ ; we have $e_{1}^{2} = v_{2} + v_{3} + v_{4}$ and $0 = e_{1}^{2}e_{1}^{2} = v_{2}^{2} + v_{3}^{2} + v_{4}^{2}$ implies $v_{4}^{2} = -(v_{2}^{2} + v_{3}^{2})$.
We set $v_{5} = v_{2}^{2} = a_{12}^{2}a_{25}e_{5}$. There is $\alpha \in F^{*}$ such that $v_{3}^{2} = \alpha v_{2}^{2} = \alpha v_{5}$.

The family $\{e_{1}, v_{2}, v_{3}, v_{4}, v_{5}\}$ is a natural basis of $N$ and its multiplication table is defined by: $e_{1}^{2} =  v_{2} + v_{3} + v_{4}, v_{2}^{2} =  v_{5}, v_{3}^{2} =  \alpha v_{5}, v_{4}^{2}  =  -(1 + \alpha)v_{5}$, and $v_{5}^{2} = 0$ with $\alpha(1 + \alpha) \neq 0$.

Let's show that $N$ is isomorphic to $N_{5,10}(\beta)$ for some nonzero $\beta$.
We set $w_{3} = v_{3} + v_{4}$ and $w_{4} = av_{3} + bv_{4}$ ; we have: $0 = w_{3}w_{4} = av_{3}^{2} + bv_{4}^{2} = [\alpha a - (1 + \alpha)b]v_{5}$ then $b = \frac{\alpha a}{1 + \alpha}.$ For $a = 1$, $b = \frac{\alpha}{1 + \alpha}$ and $w_{4} = v_{3} + \frac{\alpha}{1 + \alpha}v_{4}$.
The family $\{e_{1}, v_{2}, w_{3}, w_{4}, v_{5}\}$ is a natural basis of $N$ and its multiplication table is defined by:
$e_{1}^{2} =  v_{2} + w_{3}, v_{2}^{2} =  v_{5}, w_{3}^{2} = -v_{5}, w_{4}^{2} = [\alpha - \frac{\alpha^{2}}{1 + \alpha}]v_{5} = \frac{\alpha}{1 + \alpha}v_{5}, v_{5}^{2} = 0$ with $\frac{\alpha}{1 + \alpha} \neq 0$.
We deduce that $N \simeq N_{5, 10}\big(\be\big)$ where $\be=\frac{\alpha}{1 + \alpha}$.

2) $\dim(\ann(N)) = 2.$ We have
$e_{1}^{2} = a_{12}e_{2} + a_{13}e_{3} + a_{14}e_{4} + a_{15}e_{5}, e_{2}^{2} = a_{24}e_{4} + a_{25}e_{5}, e_{3}^{2} =  a_{34}e_{4} + a_{35}e_{5}, e_{4}^{2} = e_{5}^{2} = 0$ with $a_{12}a_{13} \neq 0$, $(a_{24}, a_{25}) \neq 0$ and $(a_{34}, a_{35}) \neq 0$.

We set $v_{2} = a_{12}e_{2} + a_{15}e_{5}$ and $v_{3} = a_{13}e_{3} + a_{14}e_{4}$. We have: $e_{1}^{2} = v_{2} + v_{3}$ and $0 = e_{1}^{2}e_{1}^{2} = v_{2}^{2} + v_{3}^{2}$ implies $v_{3}^{2} = -v_{2}^{2}$. Let's set $v_{4} = v_{2}^{2} = a_{12}^{2}e_{2}^{2}$, so $v_{3}^{2} = -v_{4}$.
Then $N = \big<e_{1}, v_{2}, v_{3}, v_{4}\big>\oplus\big<v_{5}\big>$ direct sum of evolution algebras where $\ann(N) = Kv_{4}\oplus Kv_{5}.$
\end{proof}

\subsubsection{Six-dimensional classification}

\begin{theo} Let $N$ be a nil indecomposable power-associative which is not associative evolution algebra of dimension six. Then $N$ is isomorphic to one and only one of the seven algebras  $N_{6,19}(\al,\be)$ to  $N_{6,26}$ in Table~3.

\end{theo}

\begin{proof}Since  $\dim(N) = 6$, we have $\dim(\ann(N)) < \frac{1}{2}\dim(N) = 3$ and  $\dim(\ann(N)) = 1, 2.$

1) $\dim(\ann(N)) = 1.$
\begin{eqnarray*}
e_{1}^{2} & = &  a_{12}e_{2} + a_{13}e_{3} + a_{14}e_{4} + a_{15}e_{5} + a_{16}e_{6}, e_{2}^{2} = a_{26}e_{6}, e_{3}^{2} = a_{36}e_{6}\\
e_{4}^{2} & = &  a_{42}e_{2} + a_{43}e_{3} + a_{45}e_{5} + a_{46}e_{6}, e_{5}^{2} =  a_{52}e_{2} + a_{53}e_{3} + a_{54}e_{4} + a_{56}e_{6}\\
e_{6}^{2} & = &  0 \textrm{ with }a_{12}a_{13}a_{26}a_{36} \neq 0  \textrm{, }e_{4}^{2} \neq 0 \textrm{ and }e_{5}^{2} \neq 0.
\end{eqnarray*}
Since $0 = (e_{1}^{2}e_{4})e_{i} = a_{14}a_{4i}e_{i}^{2}$, then $a_{14}a_{4i} = 0$ (with $i = 2, 3, 5$) and  $0 = (e_{1}^{2}e_{5})e_{j} = a_{15}a_{5j}e_{j}^{2}$ implies $a_{15}a_{5j} = 0$ (with $j = 2, 3, 4$),
we then distinguish four cases:

1.1) $a_{14} = a_{15} = 0$ i.e. $e_{1}^{2} = a_{12}e_{2} + a_{13}e_{3} + a_{16}e_{6}$.
We set $v_{2} = a_{12}e_{2} + a_{16}e_{6}$ and $v_{3} = a_{13}e_{3}$. We have $e_{1}^{2} = v_{2} + v_{3}$ and
$0 = e_{1}^{2}e_{1}^{2} = v_{2}^{2} + v_{3}^{2}$
implies $v_{3}^{2} = -v_{2}^{2}$.
Let's set $v_{6} = v_{2}^{2} = a_{12}^{2}a_{26}e_{6}$ ; so $v_{3}^{2} = -v_{6}$ and there are scalars $\alpha_{42}, \alpha_{43}, \alpha_{45}, \alpha_{46}$
not all zero such that $e_{4}^{2} = \alpha_{42}v_{2} + \alpha_{43}v_{3} + \alpha_{45}e_{5} + \alpha_{46}v_{6}$. Similarly, there are scalars
  $\alpha_{52}, \alpha_{53}, \alpha_{54}, \alpha_{56}$ not all zero such that
   $e_{5}^{2} = \alpha_{52}v_{2} + \alpha_{53}v_{3} + \alpha_{54}e_{4} + \alpha_{56}v_{6}$. Thus, we have
\begin{eqnarray*}
0 & = & e_{1}^{2}e_{4}^{2} = \alpha_{42}v_{2}^{2} + \alpha_{43}v_{3}^{2} = (\alpha_{42} - \alpha_{43})v_{6} \Longrightarrow \alpha_{42} = \alpha_{43}.\\
0 & = & e_{1}^{2}e_{5}^{2} = \alpha_{52}v_{2}^{2} + \alpha_{53}v_{3}^{2} = (\alpha_{52} - \alpha_{53})v_{6} \Longrightarrow \alpha_{52} = \alpha_{53}.\\
0 & = & e_{4}^{2}e_{4}^{2} = (\alpha_{42}^{2} - \alpha_{43}^{2})v_{6} + \alpha_{45}e_{5}^{2} = \alpha_{45}e_{5}^{2} \Longrightarrow \alpha_{45} = 0.\\
0 & = & e_{5}^{2}e_{5}^{2} = (\alpha_{52}^{2} - \alpha_{53}^{2})v_{6} + \alpha_{54}e_{4}^{2} = \alpha_{54}e_{4}^{2} \Longrightarrow \alpha_{54} = 0.
\end{eqnarray*}
These relations ensure power associativity of $N.$
The family $\{e_{1}, v_{2}, v_{3}, e_{4}, e_{5}, v_{6}\}$ is a natural basis of $N$ and its multiplication table is given by:
$e_{1}^{2} =  v_{2} + v_{3}, v_{2}^{2} =  v_{6}, v_{3}^{2} = -v_{6},
e_{4}^{2} = \alpha_{42}(v_{2} + v_{3}) + \alpha_{46}v_{6}, e_{5}^{2} = \alpha_{52}(v_{2} + v_{3}) + \alpha_{56}v_{6},
v_{6}^{2} =  0$ with $(\alpha_{42}, \alpha_{46}) \neq 0$ and $(\alpha_{52}, \alpha_{56}) \neq 0$.

We have: $\ann(N) = \big<v_{6}\big>$, $v_{2}, v_{3} \in \ann^{2}(N)$ and $e_{1} \in \ann^{3}(N)$.
Since $e_{4}^{2}, e_{5}^{2} \in \big<v_{2}, v_{3}, v_{6}\big>$
then the possible types of $N$ are: $[1, 4, 1]$, $[1, 3, 2]$ or $[1, 2, 3].$

1.1.1) Type $[1, 4, 1]$ leads to $\alpha_{42} = \alpha_{52} = 0$. we have $\ann^{2}(N) = \big<v_{2}, v_{3}, e_{4}, e_{5}, v_{6}\big>$ and $\ann^{3}(N) = N$.
So $\U_{3}\oplus \U_{1} = \big<e_{1}, v_{6}\big>$ and $(\U_{3}\oplus \U_{1})^{2} = \big<v_{2} + v_{3}\big>$.

$N \simeq N_{6,19}(\alpha, \beta): e_{1}^{2} =  e_{2} + e_{3}, e_{2}^{2} = e_{6}, e_{3}^{2} = -e_{6}, e_{4}^{2} = \alpha e_{6}, e_{5}^{2} =
  \beta e_{6}, e_{6}^{2} =  0$ with $\alpha, \beta \in F^{*}$.

1.1.2) Type $[1, 3, 2]$ leads to $\alpha_{42} \neq 0$ and $\alpha_{52} = 0$ or $\alpha_{42} = 0$ and $\alpha_{52} \neq 0.$

$a)$ For $\alpha_{42} \neq 0$ and $\alpha_{52} = 0$, we have: $\ann^{2}(N) = \big<v_{2}, v_{3}, e_{5}, v_{6}\big>$ and $\ann^{3}(N) = N.$
So $\U_{3}\oplus \U_{1} = \big<e_{1}, e_{4}, v_{6}\big>$ and $\big(\U_{3}\oplus \U_{1}\big)^{2} = \big<v_{2} + v_{3}, e_{4}^{2}\big>$. Thus
$\dim\big((\U_{3}\oplus \U_{1})^{2}\big) = 1, 2$.

\begin{itemize}
  \item $\dim\big((\U_{3}\oplus \U_{1})^{2}\big) = 1 \Longrightarrow \alpha_{46} = 0.$\\
$N \simeq N_{6,20}(\alpha, \beta): e_{1}^{2} =  e_{2} + e_{3}, e_{2}^{2} =  e_{6}, e_{3}^{2} = -e_{6}, e_{4}^{2} = \alpha(e_{2} + e_{3}),
e_{5}^{2}= \beta e_{6}, e_{6}^{2} = 0$ with $\alpha, \beta \in F^{*}$.
  \item $\dim\big((\U_{3}\oplus \U_{1})^{2}\big) = 2 \Longrightarrow \alpha_{46} \neq 0.$\\
$N \simeq N_{6,21}(\alpha, \beta, \gamma): e_{1}^{2} =  e_{2} + e_{3}, e_{2}^{2} = e_{6}, e_{3}^{2} =  -e_{6},
e_{4}^{2} =  \alpha(e_{2} + e_{3}) + \beta e_{6}, e_{5}^{2} =  \gamma e_{6}, e_{6}^{2} =  0$ with $\alpha, \beta, \gamma \in F^{*}$.
\end{itemize}

$b)$ For $\alpha_{42} = 0$ and $\alpha_{52} \neq 0$, we find again case $a)$ by permuting vectors $e_{4}$ and $e_{5}$.

1.1.3) Type $[1, 2, 3]$ leads to $\alpha_{42}\alpha_{52} \neq 0$. We have $\ann^{2}(N) = \big<v_{2}, v_{3}, v_{6}\big>$ and $\ann^{3}(N) = N$.
So $\U_{3}\oplus \U_{1} = \big<e_{1}, e_{4}, e_{5}, v_{6}\big>$ and $\big(\U_{3}\oplus \U_{1}\big)^{2} = \big<v_{2} + v_{3}, e_{4}^{2}, e_{5}^{2}\big>.$
Thus, $\dim\big((\U_{3}\oplus \U_{1})^{2}\big) = 1, 2.$

\begin{itemize}
  \item $\dim\big((\U_{3}\oplus \U_{1})^{2}\big) = 1 \Longrightarrow \alpha_{46} = \alpha_{56} = 0$.\\
$N \simeq N_{6,22}(\alpha, \beta): e_{1}^{2} =  e_{2} + e_{3},
e_{2}^{2} =  e_{6}, e_{3}^{2} =  -e_{6}, e_{4}^{2} =  \alpha(e_{2} + e_{3}),
e_{5}^{2} = \beta(e_{2} + e_{3}),
e_{6}^{2} =  0$ with  $\alpha, \beta \in F^{*}$.
  \item $\dim\big((\U_{3}\oplus \U_{1})^{2}\big) = 2 \Longrightarrow (\alpha_{46}, \alpha_{56}) \neq 0$.
We have $N(\alpha, \beta, \gamma, \delta): e_{1}^{2} =  e_{2} + e_{3}, e_{2}^{2} =  e_{6},
e_{3}^{2} =  -e_{6}, e_{4}^{2} =  \alpha(e_{2} + e_{3}) + \beta e_{6}, e_{5}^{2} =  \gamma(e_{2} + e_{3}) + \delta e_{6}, e_{6}^{2} = 0$
with $\alpha\gamma \neq 0$ and $(\beta, \delta) \neq 0$.\\
We distinguish four cases:
\begin{description}
 \item[$i)$] $\be\de\ne0$ and $\al\de-\be\g\ne0$ gives algebra $N_{6,23}(\al,\be,\g,\de)$.
\item[$ii)$] $\delta=0$ and $\beta\ne0$ gives $N\simeq N_{6,24}(\alpha, \beta, \gamma)$.
 \item[$iii)$] $\beta=0$ and $\delta\ne0$. By permuting $e_4$ and $e_5$ we come back to case $ii)$.
 \item[$iv)$] $\beta\delta\ne0$ and $\al\de-\be\g=0$. Since $\de=\al^{-1}\be\g$, we have $e_4^2=\al(e_2+e_3+\al^{-1}\be e_6)$ and
  $e_5^2=\g(e_2+e_3+\al^{-1}\be e_6)$. Let's do the basis change $w_{1} = e_{4}$, $w_{2} = \alpha e_{2}$, $w_{3} = \alpha e_{3} + \beta e_{6}$,  $w_{4} = e_{1}$, $w_{5} = e_{5}$ and $w_{6} = \alpha^{2}e_{6}$. In the natural basis $\{w_1,w_2,w_3,w_4,w_5,w_6\}$ the multiplication table leads to
   $N_{6,24}(\al^{-1},-\al^{-3}\be,\al^{-1}\g)$.
\end{description}
\end{itemize}

1.2) $a_{14} \neq 0$ and $a_{15} = 0$ i.e. $e_{1}^{2} = a_{12}e_{2} + a_{13}e_{3} + a_{14}e_{4} + a_{16}e_{6}$ and $e_{4}^{2} = a_{46}e_{6}$ with $a_{12}a_{13}a_{14}a_{46} \neq 0$.
Let's set $v_{2} = a_{12}e_{2} + a_{16}e_{6}$ ; $v_{3} = a_{13}e_{3}$ and $v_{4} = a_{14}e_{4}$ ; we have: $e_{1}^{2} = v_{2} + v_{3} + v_{4}$ and $0 = e_{1}^{2}e_{1}^{2} = v_{2}^{2} + v_{3}^{2} + v_{4}^{2}$ then $v_{4}^{2} = -(v_{2}^{2} + v_{3}^{2})$.
Let's set $v_{6} = v_{2}^{2} = a_{12}^{2}a_{26}e_{6}$. There is $\alpha \in F^{*}$ such that $v_{3}^{2} =  \alpha v_{2}^{2} = \alpha v_{6}$. Thus, $v_{4}^{2} = -(1 + \alpha)v_{6}$ with $1 + \alpha \neq 0.$\\
There are scalars $\alpha_{52}, \alpha_{53}, \alpha_{54}, \alpha_{56}$ not all zero such that $e_{5}^{2} = \alpha_{52}v_{2} + \alpha_{53}v_{3} + \alpha_{54}v_{4} + \alpha_{56}v_{6}.$\\
We have $0 = e_{1}^{2}e_{5}^{2} = [\alpha_{52} + \alpha\alpha_{53} - (1 + \alpha)\alpha_{54}]v_{6}$ implies $\alpha_{52} = -\alpha\alpha_{53} + (1 + \alpha)\alpha_{54}$.\\ We have
$0 = e_{5}^{2}e_{5}^{2} = [\alpha_{52}^{2} + \alpha\alpha_{53}^{2} - (1 + \alpha)\alpha_{54}^{2}]v_{6}$ leads to $0 = \alpha_{52}^{2} + \alpha\alpha_{53}^{2} - (1 + \alpha)\alpha_{54}^{2}$. We have:
$\alpha_{52}^{2} + \alpha\alpha_{53}^{2} - (1 + \alpha)\alpha_{54}^{2} = (1 + \alpha)^{2}\alpha_{54}^{2} + \alpha^{2}\alpha_{53}^{2} - 2\alpha(1 + \alpha)\alpha_{53}\alpha_{54} + \alpha\alpha_{53}^{2} - (1 + \alpha)\alpha_{54}^{2}
 = \alpha(1 + \alpha)(\alpha_{54} - \alpha_{53})^{2}$.
So, $0 = \alpha_{52}^{2} + \alpha\alpha_{53}^{2} - (1 + \alpha)\alpha_{54}^{2} = \alpha(1 + \alpha)(\alpha_{54} - \alpha_{53})^{2} \Longrightarrow \alpha_{54} = \alpha_{53}$.
Moreover $\alpha_{52} = (1 + \alpha)\alpha_{54} - \alpha\alpha_{53} = \alpha_{54}$. Thus,
\begin{eqnarray*}
e_{1}^{2} & = &  v_{2} + v_{3} + v_{4},\quad v_{2}^{2} =  v_{6},\quad v_{3}^{2} = \alpha v_{6},\quad
v_{4}^{2} = -(1 + \alpha)v_{6}\\
e_{5}^{2} & = &  \alpha_{52}(v_{2} + v_{3} + v_{4}) + \alpha_{56}v_{6},\quad
v_{6}^{2} =  0 \textrm{ with } \alpha, 1 + \alpha \in F^{*} \textrm{ and }(\alpha_{52}, \alpha_{56}) \neq 0.
\end{eqnarray*}

Let's show that $N$ is isomorphic to one of the following algebras:  $N_{6, 18}(\alpha', \beta')$, $N_{6,20}(\alpha', \beta')$,  $N_{6,21}(\alpha', \beta', \gamma')$ for some $(\alpha', \beta', \gamma')$ with $\alpha', \beta', \gamma' \in F^{*}.$

Let's set $w_{3} = v_{3} + v_{4},$ $w_{4} = e_{5}$ and $w_{5} = av_{3} + bv_{4}$ ; we have: $0 = w_{3}w_{5} = \big(\alpha a - (1 + \alpha)b\big)v_{6}$ implies $b = \frac{\alpha a}{1 + \alpha}$.
For $a = 1,$ $b = \frac{\alpha }{1 + \alpha}$ and $w_{5} = v_{3} + \frac{\alpha }{1 + \alpha}v_{4}.$

The family $\{e_{1}, v_{2}, w_{3}, w_{4}, w_{5}, v_{6}\}$ is a natural basis of $N$ and its multiplication table is defined by: $e_{1}^{2}= v_{2} + w_{3}, v_{2}^{2} =  v_{6}, w_{3}^{2} =  -v_{6}, w_{4}^{2} = \alpha_{52}(v_{2} + w_{3}) + \alpha_{56}v_{6}, w_{5}^{2} = [\alpha - \frac{\alpha^{2}}{1 + \alpha}]v_{6} = [\frac{\alpha}{1 + \alpha}]v_{6}, v_{6}^{2} =  0$ with $\frac{\alpha}{1 + \alpha} \neq 0$ and $(\alpha_{52}, \alpha_{56}) \neq 0$.

For $\alpha_{52} = 0$, then $N \simeq N_{6,19}(\alpha_{56}, \frac{\alpha}{1 + \alpha})$ with $\alpha_{56}, \frac{\alpha}{1 + \alpha} \in F^{*}$.

For $\alpha_{56} = 0$, then $N \simeq N_{6,20}(\alpha_{52}, \frac{\alpha}{1 + \alpha})$ with $\alpha_{56}, \frac{\alpha}{1 + \alpha} \in F^{*}$.

For $\alpha_{52}\alpha_{56} \neq 0$, then $N \simeq N_{6,21}(\alpha_{52}, \alpha_{56}, \frac{\alpha}{1 + \alpha})$ with $\alpha_{52}, \alpha_{56}, \frac{\alpha}{1 + \alpha} \in F^{*}$.

1.3) $a_{14} = 0$ and $a_{15} \neq 0$ i.e. $e_{1}^{2} = a_{12}e_{2} + a_{13}e_{3} + a_{15}e_{5} + a_{16}e_{6}$ and $e_{5}^{2} = a_{56}e_{6}$ with $a_{12}a_{13}a_{15}a_{56} \neq 0$.
By permuting the vectors $e_{4}$ and $e_{5}$ of natural basis, we find again case 1.2

1.4) $a_{14}a_{15} \neq 0$ i.e. $e_{1}^{2} = a_{12}e_{2} + a_{13}e_{3} + a_{14}e_{4} + a_{15}e_{5} + a_{16}e_{6},$ $e_{4}^{2} = a_{46}e_{6}$ and $e_{5}^{2} = a_{56}e_{6}$ with $a_{12}a_{13}a_{14}a_{15}a_{46}a_{56} \neq 0$. Let's set $v_{2} = a_{12}e_{2} + a_{16}e_{6},$ $v_{3} = a_{13}e_{3},$ $v_{4} = a_{14}e_{4}$ and $v_{5} = a_{15}e_{5}$ ; we have: $e_{1}^{2} = v_{2} + v_{3} + v_{4} + v_{5}$ and $0 = e_{1}^{2}e_{1}^{2} = v_{2}^{2} + v_{3}^{2} + v_{4}^{2} + v_{5}^{2}$ implies $v_{5}^{2} = -(v_{2}^{2} + v_{3}^{2} + v_{4}^{2}).$\\
Let's set $v_{6} = v_{2}^{2} = a_{12}^{2}a_{26}e_{6}$. There are $\alpha, \beta \in F^{*}$ such that $v_{3}^{2} = \alpha v_{6}$ and $v_{4}^{2} = \beta v_{6}.$
Thus, $v_{5}^{2} = -(1 + \alpha + \beta)v_{6}$ with $1 + \alpha + \beta \neq 0.$\\
The family $\{e_{1}, v_{2}, v_{3}, v_{4}, v_{5}, v_{6}\}$ is a natural basis of $N$ and its multiplication table is defined by:
$e_{1}^{2} =  v_{2} + v_{3} + v_{4} + v_{5}, v_{2}^{2} =  v_{6}, v_{3}^{2} =  \alpha v_{6}, v_{4}^{2} =  \beta v_{6}, v_{5}^{2} =  -(1 + \alpha + \beta)v_{6}, v_{6}^{2} = 0$ with $\alpha\beta(1 + \alpha + \beta) \neq 0$.

Let's show that there are vectors $x = x_{3}v_{3} + x_{4}v_{4} + x_{5}v_{5}$ and $y = y_{3}v_{3} + y_{4}v_{4} + y_{5}v_{5}$ such that the family $\{e_{1}, v_{2}, v_{3} + v_{4} + v_{5}, x, y, v_{6}\}$ is a natural basis of $N$. we set $w_{3} =  v_{3} + v_{4} + v_{5}$.
The family $\{x,y\}$ is orthogonal to $w_3$ if and only if
\begin{eqnarray}\label{Eq10}
x_{5} &=& (1 + \alpha + \beta)^{-1}(x_{3}\alpha + x_{4}\beta),\\
y_{5} &=& (1 + \alpha + \beta)^{-1}(y_{3}\alpha + y_{4}\beta).
\end{eqnarray}
Then $x$ and $y$ are written
\begin{eqnarray*}
 x &=& x_3[v_3+\al(1+\al+\be)^{-1}v_5]+x_4[v_4+\be(1+\al+\be)^{-1}v_5],\\
 y &=& y_3[v_3+\al(1+\al+\be)^{-1}v_5]+y_4[v_4+\be(1+\al+\be)^{-1}v_5].
\end{eqnarray*}
The orthogonality of $x$ and $y$ leads to $\al(1+\be)x_3y_3+\be(1+\al)x_4y_4-\al\be(x_3y_4+x_4y_3)=0$, so
\begin{equation}\label{Eq12}
 x_3\al[(1+\be)y_3-\be y_4]+x_4\be[(1+\al)y_4-\al y_3]=0.
\end{equation}

We distinguish three cases :
\begin{description}
 \item[$i)$] if $\al+1\ne 0$, by taking $x_3=0$ and $y_3=\al^{-1}(1+\al)y_4$, it comes that
 \begin{eqnarray*}
  x &=&x_4[v_4+\be(1+\al+\be)^{-1}v_5],\\
  y &=& y_4[\al^{-1}(1+\al)v_3+v_4+v_5].
 \end{eqnarray*}
 \item[$ii)$] If $\be+1\ne 0$, by taking $x_4=0$ et $y_4=\be^{-1}(1+\be)y_3$, it comes that
\begin{eqnarray*}
  x &=&x_3[v_3+\al(1+\al+\be)^{-1}v_5],\\
  y &=& y_3[v_3+ \be^{-1}(1+\be)v_4+v_5].
 \end{eqnarray*}
\end{description}
\item[$iii)$] $\al=\be=-1$. Identity \eqref{Eq12} leads to $x_3y_4+x_4y_3=0$ and we must have $x_3y_4-x_4y_3=2x_3y_4\ne0$. Then
\begin{eqnarray*}
  x &=&x_3[v_3+v_5]+x_4[v_4+v_5],\\
  y &=& y_3[v_3+v_5]+y_4[v_4+v_5].
 \end{eqnarray*}

In each of the three cases, there are $\alpha', \beta' \in F^{*}$ such that $x^{2} = \alpha'v_{6}$ and $y^{2} = \beta'v_{6}$. Thus, the multiplication table of the natural basis $\{e_{1}, v_{2}, v_{3} + v_{4} + v_{5}, x, y, v_{6}\}$ of $N$ is given by:
$e_{1}^{2} = v_{2} + w_{3}$, $v_{2}^{2} = v_{6}$, $w_{3}^{2} = -v_{6}$, $x^{2} = \alpha'v_{6}$, $y^{2} = \beta'v_{6}, v_{6}^{2} = 0$ with $\alpha'\beta' \neq 0$. So
$N$ is isomorphic to $N_{6,19}(\alpha', \beta')$.

2) $\dim(\ann(N)) = 2.$ We have:
\begin{eqnarray*}
e_{1}^{2} & = &  a_{12}e_{2} + a_{13}e_{3} + a_{14}e_{4} + a_{15}e_{5} + a_{16}e_{6}, \quad e_{2}^{2} = a_{25}e_{5} + a_{26}e_{6},\\
e_{3}^{2} & = &  a_{35}e_{5} + a_{36}e_{6}, \quad e_{4}^{2} = a_{42}e_{2} + a_{43}e_{3} + a_{45}e_{5} + a_{46}e_{6},\\
e_{5}^{2} & = &  0, \quad e_{6}^{2} =  0 \textrm{ with } a_{12}a_{13} \neq 0 \textrm{ ; }(a_{25}, a_{26}) \neq 0 \textrm{ ; }(a_{35}, a_{36}) \textrm{ and }e_{4}^{2} \neq 0.
\end{eqnarray*}
We have $0 = (e_{1}^{2}e_{4})e_{2} = a_{14}a_{42}e_{2}^{2}$ and $0 = (e_{1}^{2}e_{4})e_{3} = a_{14}a_{43}e_{3}^{2}$, so $a_{14}a_{42} = a_{14}a_{43} = 0.$

We then distinguish two cases

2.1) $a_{14} = 0$ i.e. $e_{1}^{2} = a_{12}e_{2} + a_{13}e_{3} + a_{15}e_{5} + a_{16}e_{6}$ with $a_{12}a_{13} \neq 0.$
Let's set $v_{2} = a_{12}e_{2} + a_{15}e_{5} + a_{16}e_{6}$ and $v_{3} = a_{13}e_{3}$, we have  $e_{1}^{2} = v_{2} + v_{3}$ and $0 = e_{1}^{2}e_{1}^{2} = v_{2}^{2} + v_{3}^{2}$ leads to  $v_{3}^{2} = -v_{2}^{2}.$\\
Let's set $v_{5} = v_{2}^{2} = a_{12}^{2}e_{2}^{2}$. So $v_{3}^{2} = -v_{5}$ and there are scalars $\alpha,$ $\beta,$ $\gamma$ et $\delta$ not all zero such that $e_{4}^{2} = \alpha v_{2} + \beta v_{3} + \gamma v_{5} + \delta e_{6}$. By setting $w = \gamma v_{5} + \delta e_{6}$, we have $e_{4}^{2} = \alpha v_{2} + \beta v_{3} + w$.

Moreover $0 = e_{1}^{2}e_{4}^{2} = \alpha v_{2}^{2} + \beta v_{3}^{2} = (\alpha - \beta)v_{5}$ leads to $\alpha = \beta$. This relation ensure the power associativity of $N$.

2.1.1) The family $\{v_{5}, w\}$ is linearly independent.
In that case, the family $\{e_{1}, v_{2}, v_{3}, e_{4}, v_{5}, w\}$ is a natural basis of $N$ and its multiplication table is defined by \begin{eqnarray*}
e_{1}^{2} & = &  v_{2} + v_{3}, v_{2}^{2} =  v_{5}, v_{3}^{2} = -v_{5},\\
e_{4}^{2} & = &  \alpha(v_{2} + v_{3}) + w, v_{5}^{2} = 0, w^{2} = 0 \textrm{ with } \alpha \in F.
\end{eqnarray*}

If $\alpha = 0$, then $N = \big<e_{1}, v_{2}, v_{3}, v_{5}\big>\oplus\big<e_{4}, w\big>$, direct sum of evolution algebras.

We deduce that $\alpha \neq 0$ and $N$ is isomorphic to

$N_{6,25}(\alpha): e_{1}^{2} =  e_{2} + e_{3}$, $e_{2}^{2} = e_{5}$
$e_{3}^{2} =  -e_{5}$, $e_{4}^{2} = \alpha(e_{2} + e_{3}) + e_{6}$, $e_{5}^{2} = 0$, $e_{6}^{2} = 0$ with $\alpha \in F^{*}$ and its type is $[2, 2, 2]$.
We have $\ann(N) = \big<e_{5}, e_{6}\big>$ ; $\ann^{2}(N) = \big<e_{2}, e_{3}, e_{5}, e_{6}\big>$ and $\ann^{3}(N) = N$.
So, $\U_{3}\oplus U_{1} = \big<e_{1}, e_{4}, e_{6}\big>$ and $(U_{3}\oplus U_{1})^{2} = \big<e_{2} + e_{3}, e_{6}\big>.$

2.1.2) The family $\{v_{5}, w\}$ is linearly dependent. Then, there is $\gamma \in F$ such that $w = \gamma v_{5}$. In this case $N = \big<e_{1}, v_{2}, v_{3}, v_{5}, e_{4}\big>\oplus\big<v_{6}\big>$ direct sum of evolution algebras where $\ann(N) = Kv_{5}\oplus Kv_{6}$.

2.2) $a_{14} \neq 0$ i.e. $e_{1}^{2} = a_{12}e_{2} + a_{13}e_{3} + a_{14}e_{4} + a_{15}e_{5} + a_{16}e_{6}$ and $e_{4}^{2} = a_{45}e_{5} + a_{46}e_{6}$ with $a_{12}a_{13}a_{14} \neq 0$ and $(a_{45}, a_{46}) \neq 0.$
We set $v_{2} = a_{12}e_{2} + a_{15}e_{5} + a_{16}e_{6}$, $v_{3} = a_{13}e_{3}$ and $v_{4} = a_{14}e_{4}$.
we have $e_{1}^{2} = v_{2} + v_{3} + v_{4}$ and $0 = e_{1}^{2}e_{1}^{2} =  v_{2}^{2} + v_{3}^{2} + v_{4}^{2}$ leads to $v_{4}^{2} = -(v_{2}^{2} + v_{3}^{2})$ and we set $v_{5} = v_{2}^{2} = a_{12}^{2}e_{2}^{2}.$

2.2.1) The family $\{v_{5}, v_{3}^{2}\}$ is linearly independent. By setting $v_{6} = v_{3}^{2} = a_{13}^{2}e_{3}^{2}$, we have
$N\simeq N_{6,26}: e_{1}^{2} =  e_{2} + e_{3} + e_{4}$, $e_{2}^{2} = e_{5}$, $e_{3}^{2} =  e_{6}$, $e_{4}^{2} =  -(e_{5} + e_{6})$, $e_{5}^{2} =  0$, $e_{6}^{2}= 0$ and its type is  $[2, 3, 1]$.
We have: $\ann(N) = \big<e_{5}, e_{6}\big>$ ; $\ann^{2}(N) = \big<e_{2}, e_{3}, e_{4}, e_{5}, e_{6}\big>$ and $\ann^{3}(N) = N.$
So, $\U_{3}\oplus \U_{1} = \big<e_{1}, e_{6}\big>$ and $(\U_{3}\oplus \U_{1})^{2} = \big<e_{2} + e_{3} + e_{4}\big>.$

2.2.2) The family is $\{v_{5}, v_{3}^{2}\}$ is linearly dependent. There is $\alpha \in F$ such that $v_{3}^{2} = \alpha v_{5}.$\\
We have $v_{4}^{2} = -(1 + \alpha)v_{5}$ and $N = \big<e_{1}, v_{2}, v_{3}, v_{4}, v_{5}\big>\oplus\big<v_{6}\big>$ direct sum of evolution algebras where $\ann(N) = Kv_{5}\oplus Kv_{6}.$
\end{proof}

\bigskip
\subsection{General classification}

Let $\E_{i,j}$ be the $j$-th power-associative evolution algebra of dimension $i$ over the commutative fields $F.$ If $\E_{i,j}$ is nilalgebra, we denote $N_{i,j}$. According to Definition~4.0.1 and Theorem~\ref{Wedder}, we determine $\E_{i,j}$ for $1 \leq i \leq 6$. Moreover, $\E_{ss}$ denotes the semisimple component in the Wedderburn decomposition of $\E_{i,j}$ where $s$ is the number of idempotents pairwise orthogonal.

\begin{prop}
Let $\E$ be a power-associative evolution algebra of dimension $\leq 4$. Then $\E$ is isomorphic to one and only one of the algebras in Table $1$.
{\footnotesize
\begin{longtable}[c]{|p{6mm}|c|p{6,3cm}|c|c|}
 \caption*{\underline{\textbf{Table 1:}\label{table1}} $\dim(\E)\leq 4$}\\
  \hline
$\dim$&  {$\E$} &   {\small Multiplication} &   {\small Type} &   {\small Associative}  \\\hline
\endfirsthead
 \hline
 $\dim$&  {$\E$} &   {\small Multiplication} &   {\small Type} &   {\small Associative}  \\\hline
\endhead \hline
\endfoot
\multirow{2}{6mm}{$\mathbf{1}$}&  {$N_{1,1}$} &   {$e_{1}^{2} = 0$} &   {$[1]$} &   {\small Yes}  \\\cline{2-5}
  &{$\E_{1,2} = \E_{11}$} &   {$e_{1}^{2} = e_{1}$} & $-$ &   {\small Yes}  \\
 \hline
 \multirow{4}{6mm}{$\mathbf{2}$}&{$N_{2,1} = N_{1,1}\oplus N_{1,1}$} &   {$e_{1}^{2} = e_{2}^{2} = 0$} &   {$[2]$} &   {\small Yes}  \\\cline{2-5}
  &{$N_{2,2}$} &   {$e_{1}^{2} = e_{2},$ $e_{2}^{2} = 0$} &   {$[1, 1]$} &   {\small Yes}  \\\cline{2-5}
  &{$\E_{2,3} = \E_{11}\oplus N_{1,1}$} &   {$e_{1}^{2} = e_{1},$ $e_{2}^{2} = 0$} & $-$ &   {\small Yes}  \\\cline{2-5}
  &{$\E_{2,4} = \E_{22}$} &   {$e_{1}^{2} = e_{1},$ $e_{2}^{2} = e_{2}$} & $-$ &   {\small Yes}\\
 \hline
 \multirow{7}{6mm}{$\mathbf{3}$}&{$N_{3,1} = N_{2,1}\oplus N_{1,1}$} &   {$e_{1}^{2} = e_{2}^{2} = e_{3}^{2} = 0$} &   {$[3]$} &   {\small Yes}  \\\cline{2-5}
  &{$N_{3,2} = N_{2,2}\oplus N_{1,1}$} &   {$e_{1}^{2} = e_{2},$ $e_{2}^{2} = e_{3}^{2} = 0$} &   {$[2, 1]$} &   {\small Yes}  \\\cline{2-5}
  &{$N_{3,3}(\alpha)$} &   {$e_{1}^{2} = e_{3},$ $e_{2}^{2} = \alpha e_{3},$ $e_{3}^{2} = 0$ with $\alpha \in F^{*}$} &   {$[1, 2]$} &   {\small Yes}  \\\cline{2-5}
  &{$\E_{3,4} = \E_{11}\oplus N_{2,1}$} &   {$e_{1}^{2} = e_{1},$ $e_{2}^{2} = e_{3}^{2} = 0$} &   $-$ &   {\small Yes}  \\\cline{2-5}
  &{$\E_{3,5} = \E_{11}\oplus N_{2,2}$} &   {$e_{1}^{2} = e_{1},$ $e_{2}^{2} = e_{3},$ $e_{3}^{2} = 0$} & $-$ &   {\small Yes}  \\\cline{2-5}
  &{$\E_{3,6} = \E_{22}\oplus N_{1,1}$} &   {$e_{1}^{2} = e_{1},$ $e_{2}^{2} = e_{2},$ $e_{3}^{2} = 0$} & $-$ &   {\small Yes}  \\\cline{2-5}
  &{$\E_{3,7} = \E_{33}$} &   {$e_{1}^{2} = e_{1},$ $e_{2}^{2} = e_{2},$ $e_{3}^{2} = e_{3}$} & $-$ &   {\small Yes}  \\
 \hline
\multirow{13}{6mm}{$\mathbf{4}$}&{$N_{4,1} = N_{3,1}\oplus N_{1,1}$} &   {$e_{1}^{2} = e_{2}^{2} = e_{3}^{2} = e_{4}^{2} = 0$} &   {$[4]$} &   {\small Yes}  \\\cline{2-5}
  &{$N_{4,2} = N_{3,2}\oplus N_{1,1}$} &   {$e_{1}^{2} = e_{2},$ $e_{2}^{2} = e_{3}^{2} = e_{4}^{2} = 0$} &   {$[3, 1]$} &   {\small Yes}  \\\cline{2-5}
  &{$N_{4,3}(\alpha) = N_{3,3}(\alpha)\oplus N_{1,1}$} &   {$e_{1}^{2} = e_{3},$ $e_{2}^{2} = \alpha e_{3},$ $e_{3}^{2} = e_{4}^{2} = 0$ with $\alpha \in F^{*}$} &   {$[2,2]$} &   {\small Yes}  \\\cline{2-5}
  &{$N_{4,4} = N_{2,2}\oplus N_{2,2}$} &   {$e_{1}^{2} = e_{2},$ $e_{2}^{2} = 0,$ $e_{3}^{2} = e_{4},$ $e_{4}^{2} = 0$} &   {$[2, 2]$} &   {\small Yes}  \\\cline{2-5}

  &{$N_{4,5}(\alpha, \beta)$} &   {$e_{1}^{2} = e_{4},$ $e_{2}^{2} = \alpha e_{4},$ $e_{3}^{2} = \beta e_{4},$ $e_{4}^{2} = 0$ with $\alpha, \beta \in F^{*}$} &   {$[1, 3]$} &   {\small Yes}  \\\cline{2-5}
  &{$N_{4,6}$} &   {$e_{1}^{2} = e_{2} + e_{3},$ $e_{2}^{2} = e_{4},$ $e_{3}^{2} = -e_{4},$ $e_{4}^{2} = 0$} &   {$[1, 2, 1]$} &   {\small No}  \\\cline{2-5}
  &{$\E_{4,7} = \E_{11}\oplus N_{3,1}$}  &   {$e_{1}^{2} = e_{1},$ $e_{2}^{2} = e_{3}^{2} = e_{4}^{2} = 0$} & $-$ & {Yes}  \\\cline{2-5}
  &{$\E_{4,8} = \E_{11}\oplus N_{3,2}$} &   {$e_{1}^{2} = e_{1},$ $e_{2}^{2} = e_{3},$ $e_{3}^{2} = e_{4}^{2} = 0$} & $-$ &   {\small Yes}  \\\cline{2-5}
  &{$\E_{4,9}(\alpha) = \E_{11}\oplus N_{3,3}(\alpha)$} &   {$e_{1}^{2} = e_{1},$ $e_{2}^{2} = e_{4},$ $e_{3}^{2} = \alpha e_{4},$ $e_{4}^{2} = 0$ with $\alpha \in F^{*}$} & $-$ &   {\small Yes}  \\\cline{2-5}
  &{$\E_{4,10} = \E_{22}\oplus N_{2,1}$} &   {$e_{1}^{2} = e_{1},$ $e_{2}^{2} = e_{2},$ $e_{3}^{2} = e_{4}^{2} = 0$} & $-$ &   {\small Yes } \\\cline{2-5}
  &{$\E_{4,11} = \E_{22}\oplus N_{2,2}$} &   {$e_{1}^{2} = e_{1},$ $e_{2}^{2} = e_{2},$ $e_{3}^{2} = e_{4},$ $e_{4}^{2} = 0$}  & $-$ &   {\small Yes}  \\\cline{2-5}
  &{$\E_{4,12} = \E_{33}\oplus N_{1,1}$} &   {$e_{1}^{2} = e_{1},$ $e_{2}^{2} = e_{2},$ $e_{3}^{2} = e_{3},$ $e_{4}^{2} = 0$} & $-$ &   {\small Yes}  \\\cline{2-5}
  &{$\E_{4,13} = \E_{44}$} &   {$e_{1}^{2} = e_{1},$ $e_{2}^{2} = e_{2},$ $e_{3}^{2} = e_{3},$ $e_{4}^{2} = e_{4}$} & $-$ &   {\small Yes}  \\
\hline
\end{longtable}
}
\end{prop}

\begin{prop}
Let $\E$ be a power-associative evolution algebra of dimension $5$. Then $\E$ is isomorphic to one and only one of the algebras in Table $2$.
{\footnotesize
\begin{longtable}[c]{|c|p{6,6cm}|c|c|}
 \caption*{\underline{\textbf{Table 2:}\label{Table2}} $\dim(\E)=5$}\\
  \hline
  {$\E$} &   {\small Multiplication} &   {\small Type} &   {\small Associative}  \\\hline
\endfirsthead
 \hline
 {$\E$} &   {\small Multiplication} &   {\small Type} &   {\small Associative}  \\\hline
\endhead \hline
\endfoot
  {$N_{5,1} = N_{4,1}\oplus N_{1,1}$} &   {$e_{1}^{2} = e_{2}^{2} = e_{3}^{2} = e_{4}^{2} = e_{5}^{2} = 0$} &   {$[5]$} &   {\small Yes}  \\\hline
  {$N_{5,2} = N_{4,2}\oplus N_{1,1}$} &   {$e_{1}^{2} = e_{2},$ $e_{2}^{2} = e_{3}^{2} = e_{4}^{2} = e_{5}^{2} = 0$} &   {$[4, 1]$} &   {\small Yes}  \\\hline
  {$N_{5,3}(\alpha) = N_{4,3}(\alpha)\oplus N_{1,1}$} &   {$e_{1}^{2} = e_{3},$ $e_{2}^{2} = \alpha e_{3},$ $e_{3}^{2} = e_{4}^{2} = e_{5}^{2} = 0$ with $\alpha \in F^{*}$} &   {$[3,2]$} &   {\small Yes } \\\hline
  {$N_{5,4} = N_{4,4}\oplus N_{1,1}$} &   {$e_{1}^{2} = e_{2},$ $e_{2}^{2} = 0,$ $e_{3}^{2} = e_{4},$ $e_{4}^{2} = e_{5}^{2} = 0$} &   {$[3, 2]$} &   {\small Yes}  \\\hline
  {$N_{5,5}(\alpha, \beta) = N_{4,5}(\alpha, \beta)\oplus N_{1, 1}$} &   {$e_{1}^{2} = e_{4},$ $e_{2}^{2} = \alpha e_{4},$ $e_{3}^{2} = \beta e_{4},$ $e_{4}^{2} = e_{5}^{2} = 0$ with $\alpha, \beta \in F^{*}$} &   {$[2, 3]$} &   {\small Yes}  \\\hline
  {$N_{5,6}(\alpha) = N_{3,3}(\alpha)\oplus N_{2,2}$} &   {$e_{1}^{2} = e_{3},$ $e_{2}^{2} = \alpha e_{3},$ $e_{3}^{2} = 0,$ $e_{4}^{2} =  e_{5},$ $e_{5}^{2} = 0$ with $\alpha \in F^{*}$} &   {$[2, 3]$} &   {\small Yes}  \\\hline

  {$N_{5,7} = N_{4,6}\oplus N_{1,1}$}  &   {$e_{1}^{2} = e_{2} + e_{3},$ $e_{2}^{2} = e_{4},$ $e_{3}^{2} = -e_{4},$ $e_{4}^{2} = e_{5}^{2} = 0$} &   {$[2, 2, 1]$} &   {\small No}  \\\hline
  {$N_{5,8}(\alpha, \beta, \gamma)$} &   {$e_{1}^{2} = e_{5},$ $e_{2}^{2} = \alpha e_{5},$ $e_{3}^{2} = \beta e_{5},$ $e_{4}^{2} = \gamma e_{5},$ $e_{5}^{2} = 0$ with $\alpha, \beta, \gamma \in F^{*}$} &   {$[1, 4]$} &   {\small Yes}  \\\hline
  {$N_{5,9}(\alpha, \beta)$} &   {$e_{1}^{2} = e_{4},$ $e_{2}^{2} = \alpha e_{4} + \beta e_{5},$ $e_{3}^{2} = e_{5},$ $e_{4}^{2} = e_{5}^{2} = 0$ with $\alpha, \beta \in F^{*}$} &   {$[2, 3]$} &   {\small Yes}  \\\hline
  {$N_{5,10}(\alpha)$} &   {$e_{1}^{2} = e_{2} + e_{3},$ $e_{2}^{2} = e_{5},$ $e_{3}^{2} = -e_{5},$ $e_{4}^{2} = \alpha e_{5} ,$ $e_{5}^{2} = 0$ with $\alpha \in F^{*}$} &   {$[1, 3, 1]$} &   {\small No}  \\\hline
  {$N_{5,11}(\alpha)$} &   {$e_{1}^{2} = e_{2} + e_{3},$ $e_{2}^{2} = e_{5},$ $e_{3}^{2} = -e_{5},$ $e_{4}^{2} = \alpha(e_{2} + e_{3}),$ $e_{5}^{2} = 0$ with $\alpha \in F^{*}$} &   {$[1, 2, 2]$} &   {\small No}  \\\hline
  {$N_{5,12}(\alpha, \beta)$} &   {$e_{1}^{2} = e_{2} + e_{3},$ $e_{2}^{2} = e_{5},$ $e_{3}^{2} = -e_{5},$ $e_{4}^{2} = \alpha(e_{2} + e_{3}) +  \beta e_{5},$ $e_{5}^{2} = 0$ with $\alpha, \beta \in F^{*}$} &   {$[1, 2, 2]$} &   {\small No}  \\\hline
  {$\E_{5,13} = \E_{11}\oplus N_{4,1}$} &   {$e_{1}^{2} = e_{1},$ $e_{2}^{2} = e_{3}^{2} = e_{4}^{2} = e_{5}^{2} = 0$} & $-$ &   {\small Yes}  \\\hline
  {$\E_{5,14} = \E_{11}\oplus N_{4,2}$} &   {$e_{1}^{2} = e_{1},$ $e_{2}^{2} = e_{3},$ $e_{3}^{2} = e_{4}^{2} = e_{5}^{2} = 0$} & $-$ &   {\small Yes}  \\\hline
  {$\E_{5,15}(\alpha) = \E_{11}\oplus N_{4,3}(\alpha)$} &   {$e_{1}^{2} = e_{1},$ $e_{2}^{2} = e_{4},$ $e_{3}^{2} = \alpha e_{4},$ $e_{4}^{2} = e_{5}^{2} = 0$ with $\alpha \in F^{*}$} & $-$ &   {\small Yes}  \\\hline

  {$\E_{5,16} = \E_{11}\oplus N_{4,4}$} &   {$e_{1}^{2} = e_{1},$ $e_{2}^{2} = e_{3},$  $e_{3}^{2} = 0,$ $e_{4}^{2} = e_{5},$ $e_{5}^{2} = 0$} & $-$ &   {\small Yes}  \\\hline
  {$\E_{5,17}(\alpha, \beta) = \E_{11}\oplus N_{4,5}(\alpha, \beta)$} &   {$e_{1}^{2} = e_{1},$ $e_{2}^{2} = e_{5},$ $e_{3}^{2} = \alpha e_{5},$ $e_{4}^{2} = \beta e_{5},$ $e_{5}^{2} = 0$ with $\alpha, \beta \in F^{*}$} & $-$ &   {\small Yes}  \\\hline
  {$\E_{5,18} = \E_{11}\oplus N_{4,6}$} &   {$e_{1}^{2} = e_{1},$ $e_{2}^{2} = e_{3} + e_{4},$ $e_{3}^{2} = e_{5},$ $e_{4}^{2} = -e_{5},$ $e_{5}^{2} = 0$} & $-$ &   {\small No}  \\\hline
  {$\E_{5,19} = \E_{22}\oplus N_{3,1}$} &   {$e_{1}^{2} = e_{1},$ $e_{2}^{2} = e_{2},$ $e_{3}^{2} = e_{4}^{2} = e_{5}^{2} = 0$} & $-$ &   {\small Yes}  \\\hline
  {$\E_{5,20} = \E_{22}\oplus N_{3,2}$}  &   {$e_{1}^{2} = e_{1},$ $e_{2}^{2} = e_{2},$ $e_{3}^{2} = e_{4},$ $e_{4}^{2} = e_{5}^{2} = 0$} & $-$ &   {\small Yes}  \\\hline
  {$\E_{5,21}(\alpha) = \E_{22}\oplus N_{3,3}(\alpha)$} &   {$e_{1}^{2} = e_{1},$ $e_{2}^{2} = e_{2},$ $e_{3}^{2} = e_{5},$ $e_{4}^{2} =  \alpha e_{5},$ $e_{5}^{2} = 0$ with $ \alpha \in K^{*}$} & $-$ &   {\small Yes}  \\\hline
  {$\E_{5,22} = \E_{33}\oplus N_{2,1}$} &   {$e_{1}^{2} = e_{1},$ $e_{2}^{2} = e_{2},$ $e_{3}^{2} = e_{3},$ $e_{4}^{2} = e_{5}^{2} = 0$} & $-$ &   {\small Yes}  \\\hline
  {$\E_{5,23} = \E_{33}\oplus N_{2,2}$} &   {$e_{1}^{2} = e_{1},$ $e_{2}^{2} = e_{2},$ $e_{3}^{2} = e_{3},$ $e_{4}^{2} = e_{5},$ $e_{5}^{2} = 0$} & $-$ &   {\small Yes}  \\\hline
  {$\E_{5,24} = \E_{44}\oplus N_{1,1}$} &   {$e_{1}^{2} = e_{1},$ $e_{2}^{2} = e_{2},$ $e_{3}^{2} = e_{3},$ $e_{4}^{2} = e_{4},$ $e_{5}^{2} = 0$} & $-$ &   {\small Yes}  \\\hline
  {$\E_{5,25} = \E_{55}$} &   {$e_{1}^{2} = e_{1},$ $e_{2}^{2} = e_{2},$ $e_{3}^{2} = e_{3},$ $e_{4}^{2} = e_{4},$ $e_{5}^{2} = e_{5}$} & $-$ &   {\small Yes}  \\
\hline
\end{longtable}
}
\end{prop}

\begin{prop}
Let $\E$ be a power-associative evolution algebra of dimension $6$. Then $\E$ is isomorphic to one and only one of the algebras in Table $3$.
 {\footnotesize
 \begin{longtable}[c]{|c|p{6,2cm}|c|c|}
 \caption*{\underline{\textbf{Table 3:}\label{Table3}} $\dim(\E)=6$}\\
  \hline
    {$\E$} &   {\small Multiplication} &   {\small Type} & {\footnotesize Associative}  \\\hline
\endfirsthead
 \hline   {$\E$} &   {\small Multiplication} &   {\small Type} & {\footnotesize Associative}  \\\hline
\endhead \hline
\endfoot
{$N_{6,1}$ = $N_{5,1}\oplus N_{1,1}$} & {$e_{1}^{2} = e_{2}^{2} = e_{3}^{2} = e_{4}^{2} = e_{5}^{2} = e_{6}^{2} = 0$} & {$[6]$} & {\small Yes}  \\\hline
{$N_{6,2} = N_{5,2}\oplus N_{1,1}$} & {$e_{1}^{2} = e_{2},$ $e_{2}^{2} = e_{3}^{2} = e_{4}^{2} = e_{5}^{2} = e_{6}^{2} = 0$} & {$[5, 1]$ } & {\small Yes}  \\\hline
{$N_{6,3}(\alpha) = N_{5,3}(\alpha)\oplus N_{1,1}$} & {$e_{1}^{2} = e_{3},$ $e_{2}^{2} = \alpha e_{3},$ $e_{3}^{2} = e_{4}^{2} = e_{5}^{2} = e_{6}^{2} = 0$ with $\alpha \in F^{*}$} & $[4, 2]$ & {\small Yes} \\\hline
{$N_{6,4} = N_{5,4}\oplus N_{1,1}$} & {$e_{1}^{2} = e_{2},$ $e_{2}^{2} = 0,$ $e_{3}^{2} = e_{4},$ $e_{4}^{2} = e_{5}^{2} = e_{6}^{2} = 0$} & {$[4,2]$} & {\small Yes} \\\hline
{$N_{6,5}(\alpha, \beta) = N_{5,5}(\alpha, \beta)\oplus N_{1, 1}$} & {$e_{1}^{2} = e_{4},$ $e_{2}^{2} = \alpha e_{4},$ $e_{3}^{2} = \beta e_{4},$ $e_{4}^{2} = e_{5}^{2} = e_{6}^{2} = 0$ with $\alpha, \beta \in F^{*}$} & {$[3, 3]$} & {\small Yes} \\\hline
{$N_{6,6}(\alpha) = N_{5,6}(\alpha)\oplus N_{1,1}$} & {$e_{1}^{2} = e_{3},$ $e_{2}^{2} = \alpha e_{3},$ $e_{3}^{2} = 0,$ $e_{4}^{2} =  e_{5},$ $e_{5}^{2} = e_{6}^{2} = 0$ with $\alpha \in F^{*}$} & {$[3, 3]$} & {\small Yes}  \\\hline
{$N_{6,7}(\alpha, \beta, \gamma) = N_{5,8}(\alpha, \beta, \gamma)\oplus N_{1,1}$}  & {$e_{1}^{2} = e_{5},$ $e_{2}^{2} = \alpha e_{5},$ $e_{3}^{2} = \beta e_{5},$ $e_{4}^{2} = \gamma e_{5},$ $e_{5}^{2} = e_{6}^{2} = 0$ with $\alpha, \beta, \gamma \in F^{*}$} & {$[2, 4]$} & {\small Yes}  \\\hline
{$N_{6,8}(\alpha, \beta) = N_{5,9}(\alpha, \beta)\oplus N_{1,1}$} &{$e_{1}^{2} = e_{4},$ $e_{2}^{2} = \alpha e_{4} + \beta e_{5},$ $e_{3}^{2} = e_{5},$ $e_{4}^{2} = e_{5}^{2} = e_{6}^{2} = 0$ with $\alpha, \beta \in F^{*}$} &{$[3, 3]$} & {\small Yes}  \\\hline
{$N_{6,9}(\alpha, \beta) = N_{4,5}(\alpha, \beta)\oplus N_{2,2}$} & {$e_{1}^{2} = e_{4},$ $e_{2}^{2} = \alpha e_{4},$ $e_{3}^{2} = \beta e_{4},$ $e_{4}^{2} = 0,$ $e_{5}^{2} = e_{6},$ $e_{6}^{2} = 0$ with $\alpha, \beta \in F^{*}$} &{$[2, 4]$} & {\small Yes} \\\hline
{$N_{6,10}(\alpha, \beta) = N_{3,3}(\alpha)\oplus N_{3,3}(\beta)$} & {$e_{1}^{2} = e_{3},$ $e_{2}^{2} = \alpha e_{3},$ $e_{3}^{2} = 0,$ $e_{4}^{2} = e_{6},$ $e_{5}^{2} = \beta e_{6},$ $e_{6}^{2} = 0$ with $\alpha, \beta \in F^{*}$} &{$[2, 4]$} & {\small Yes}  \\\hline
{$N_{6,11} = N_{5, 7}\oplus N_{1, 1}$} &   {$e_{1}^{2} = e_{2} + e_{3},$ $e_{2}^{2} = e_{4},$ $e_{3}^{2} = -e_{4},$ $e_{4}^{2} = e_{5}^{2} = e_{6}^{2} = 0$} &   {$[3, 2, 1]$} &   {\small No}  \\\hline
{$N_{6,12}(\alpha) = N_{5, 10}(\alpha)\oplus N_{1, 1}$} &   {$e_{1}^{2} = e_{2} + e_{3},$ $e_{2}^{2} = e_{5},$ $e_{3}^{2} = -e_{5},$ $e_{4}^{2} = \alpha e_{5},$ $e_{5}^{2} = e_{6}^{2} =  0$ with $\alpha \in F^{*}$} &   {$[2, 3, 1]$} &   {\small No } \\\hline
{$N_{6,13}(\alpha) = N_{5, 11}(\alpha)\oplus N_{1, 1}$} &   {$e_{1}^{2} = e_{2} + e_{3},$ $e_{2}^{2} = e_{5},$ $e_{3}^{2} = -e_{5},$ $e_{4}^{2} = \alpha(e_{2} + e_{3}),$ $e_{5}^{2} = e_{6}^{2} =  0$ with $\alpha \in F^{*}$} &   {$[2, 2, 2]$} &   {\small No}  \\\hline
{$N_{6,14}(\alpha, \beta) = N_{5, 12}(\alpha, \beta)\oplus N_{1, 1}$} &   {$e_{1}^{2} = e_{2} + e_{3},$ $e_{2}^{2} = e_{5},$ $e_{3}^{2} = -e_{5},$ $e_{4}^{2} = \alpha(e_{2} + e_{3}) + \beta e_{5},$ $e_{5}^{2} = e_{6}^{2} =  0$ with $\alpha, \beta \in F^{*}$} &   {$[2, 2, 2]$} &   {\small No}  \\\hline
{$N_{6,15} = N_{4, 6}\oplus N_{2, 2}$} & {$e_{1}^{2} = e_{2} + e_{3},$ $e_{2}^{2} = e_{4},$ $e_{3}^{2} = -e_{4},$ $e_{4}^{2} = 0,$ $e_{5}^{2} = e_{6},$ $e_{6}^{2} = 0$} & {$[2, 3, 1]$} & {\small No} \\\hline
{$N_{6,16}(\alpha, \beta, \gamma, \delta)$} &   {$e_{1}^{2} = e_{6}, e_{2}^{2} = \alpha e_{6},$ $e_{3}^{2} = \beta e_{6},$ $e_{4}^{2} = \gamma e_{6},$ $e_{5}^{2} = \delta e_{6}$, $e_{6}^{2} = 0$ with $\alpha, \beta, \gamma, \delta \in F^{*}$} &   {$[1, 5]$} &   {\small Yes}  \\\hline
{$N_{6,17}(\alpha, \beta, \gamma)$} &   {$e_{1}^{2} = e_{5}$, $e_{2}^{2} = \alpha e_{5} +  \beta e_{6}$, $e_{3}^{2} =  \g e_{6}$, $e_{4}^{2} = e_{6}$, $e_{5}^{2} = e_{6}^{2} = 0$ with $\alpha\be\g\neq 0$}  &   {$[2, 4]$} & {\small Yes}  \\\hline
{$N_{6,18}(\alpha, \beta, \gamma,\delta)$} &   {$e_{1}^{2} = e_{5}$, $e_{2}^{2} = \al e_{5} +  \be e_{6}$, $e_{3}^{2} = \gamma e_{5}+\de e_6$, $e_{4}^{2} = e_{6}$, $e_{5}^{2} = e_{6}^{2} = 0$ with $\al\be\ne0$, $\g\de\ne 0$ and $\alpha\delta-\beta\gamma \neq 0$} &   {$[2, 4]$} &   {\small Yes}  \\\hline
{$N_{6,19}(\alpha, \beta)$} &   {$e_{1}^{2} = e_{2} + e_{3},$ $e_{2}^{2} = e_{6},$ $e_{3}^{2} = -e_{6},${} $e_{4}^{2} = \alpha e_{6},$  $e_{5}^{2} = \beta e_{6},$ $e_{6}^{2} = 0$ with $\alpha, \beta \in F^{*}$} &   {$[1, 4, 1]$} & {\small No} \\\hline
$N_{6,20}(\alpha, \beta)$ &   {$e_{1}^{2} = e_{2} + e_{3},$ $e_{2}^{2} = e_{6},$ $e_{3}^{2} = -e_{6},${} $e_{4}^{2} = \alpha(e_{2} + e_{3}),$  $e_{5}^{2} = \beta e_{6},$ $e_{6}^{2} = 0$ with $\alpha, \beta \in F^{*}$} &   {$[1, 3, 2]$} & {\small No}  \\\hline
$N_{6,21}(\alpha, \beta, \gamma)$ &   {$e_{1}^{2} = e_{2} + e_{3},$ $e_{2}^{2} = e_{6},$ $e_{3}^{2} = -e_{6},$ $e_{4}^{2} = \alpha(e_{2} + e_{3}) + \beta e_{6},$ $e_{5}^{2} = \gamma e_{6},$ $e_{6}^{2} = 0$ with $\alpha, \beta,\gamma \in F^{*}$} &   {$[1, 3, 2]$} &   {\small No} \\\hline
{$N_{6,22}(\alpha, \beta)$} &   {$e_{1}^{2} = e_{2} + e_{3},$ $e_{2}^{2} = e_{6},$ $e_{3}^{2} = -e_{6},$ $e_{4}^{2} = \alpha(e_{2} + e_{3}),$  $e_{5}^{2} = \beta(e_{2} + e_{3}),$ $e_{6}^{2} = 0$ with $\alpha, \beta \in F^{*}$} &   {$[1, 2, 3]$} &   {\small No}  \\\hline
{$N_{6,23}(\alpha, \beta, \gamma, \delta)$} &   {$e_{1}^{2} = e_{2} + e_{3},$ $e_{2}^{2} = e_{6},$  $e_{3}^{2} = -e_{6},$ \newline $e_{4}^{2} = \alpha(e_{2} + e_{3}) + \beta e_{6},$ $e_{5}^{2} = \gamma(e_{2} + e_{3}) + \delta e_{6},$ $e_{6}^{2} = 0$ with $\alpha\delta-\beta\gamma \neq 0$, $\al\g\ne0$ et $\beta\delta \neq 0$} &   {$[1, 2, 3]$} &   {\small No}  \\\hline
{$N_{6,24}(\alpha, \beta, \gamma)$} &   {$e_{1}^{2} = e_{2} + e_{3},$ $e_{2}^{2} = e_{6},$  $e_{3}^{2} = -e_{6},${} \newline $e_{4}^{2} = \alpha(e_{2} + e_{3}) + \beta e_{6},$ $e_{5}^{2} = \gamma(e_{2} + e_{3}),$ $e_{6}^{2} = 0$ with $\al\beta\gamma \neq 0$} &   {$[1, 2, 3]$} &   {\small No}  \\\hline
{$N_{6,25}(\alpha)$} &   {$e_{1}^{2} = e_{2} + e_{3},$ $e_{2}^{2} = e_{5},$ $e_{3}^{2} = -e_{5},${} $e_{4}^{2} = \alpha(e_{2} + e_{3}) + e_{6},$  $e_{5}^{2} = e_{6}^{2} = 0$ with $\alpha \in F^{*}$} &   {$[2, 2, 2]$} &   {\small No}  \\\hline
{$N_{6,26}$} &   {$e_{1}^{2} = e_{2} + e_{3} + e_{4},$ $e_{2}^{2} = e_{5},$  $e_{3}^{2} = e_{6},${} $e_{4}^{2} = -(e_{5} + e_{6}),$  $e_{5}^{2} = e_{6}^{2} = 0$} &   {$[2, 3, 1]$} &   {\small No} \\\hline
{$\E_{6,27}$ = $\E_{11}$$\oplus N_{5,1}$} &   {$e_{1}^{2} = e_{1},$ $e_{2}^{2} = e_{3}^{2} = e_{4}^{2} = e_{5}^{2} = e_{6}^{2} = 0$} & $-$ &   {\small Yes}  \\\hline
{$\E_{6,28}$ = $\E_{11}\oplus N_{5,2}$} &   {$e_{1}^{2} = e_{1},$ $e_{2}^{2} = e_{3},$ $e_{3}^{2} = e_{4}^{2} = e_{5}^{2} = e_{6}^{2} = 0$} & $-$ &   {\small Yes}  \\\hline
{$\E_{6,29}(\alpha)$ = $\E_{11}\oplus N_{5,3}(\alpha)$} &   {$e_{1}^{2} = e_{1},$ $e_{2}^{2} = e_{4},$ $e_{3}^{2} = \alpha e_{4},$ $e_{4}^{2} = e_{5}^{2} = e_{6}^{2} = 0$ with $\alpha \in F^{*}$} & $-$ &   {\small Yes}  \\\hline
{$\E_{6,30} = \E_{11}\oplus N_{5,4}$} &   {$e_{1}^{2} = e_{1},$ $e_{2}^{2} = e_{3},$ $e_{3}^{2} = 0,$ $e_{4}^{2} = e_{5},$ $e_{5}^{2} = e_{6}^{2} = 0$} & $-$ &   {\small Yes}  \\\hline
{$\E_{6,31}(\alpha, \beta) = \E_{11}\oplus N_{5,5}(\alpha, \beta)$} &   {$e_{1}^{2} = e_{1},$ $e_{2}^{2} = e_{5},$ $e_{3}^{2} = \alpha e_{5},$ $e_{4}^{2} = \beta e_{5},$ $e_{5}^{2} = e_{6}^{2} = 0$ with $\alpha, \beta \in F^{*}$} & $-$ &   {\small Yes}  \\\hline
{$\E_{6,32}(\alpha) = \E_{11}\oplus N_{5,6}(\alpha)$} &   $e_{1}^{2} = e_{1},$ $e_{2}^{2} = e_{4},$  $e_{3}^{2} =  \alpha e_{4},$ $e_{4}^{2} = 0,$ $e_{5}^{2} = e_{6},$ $e_{6}^{2} = 0$ with $\alpha \neq 0$. & $-$ &   {\small Yes}  \\\hline
  {$\E_{6,33}(\alpha, \beta, \gamma) = \E_{11}\oplus N_{5,8}(\alpha, \beta, \gamma)$} &   {$e_{1}^{2} = e_{1},$ $e_{2}^{2} = e_{6},$ $e_{3}^{2} = \alpha e_{6},$ $e_{4}^{2} = \beta e_{6},$ $e_{5}^{2} = \gamma e_{6},$ $e_{6}^{2} = 0$ with $\alpha, \beta, \gamma \in F^{*}$} & $-$ &   {\small Yes}  \\\hline
  {$\E_{6,34}(\alpha, \beta) = \E_{11}\oplus N_{5,9}(\alpha, \beta)$} &   {$e_{1}^{2} = e_{1},$ $e_{2}^{2} = e_{5},$ $e_{3}^{2} = \alpha e_{5} + \beta e_{6},$ $e_{4}^{2} = e_{6},$ $e_{5}^{2} = e_{6}^{2} = 0$ with $\alpha, \beta \in F^{*}$} & $-$ &   {\small Yes}  \\\hline
  {$\E_{6,35} = \E_{11}\oplus N_{5,7}$} &   {$e_{1}^{2} = e_{1},$ $e_{2}^{2} = e_{3} + e_{4},$ $e_{3}^{2} = e_{5},$ $e_{4}^{2} = -e_{5},$ $e_{5}^{2} = e_{6}^{2} = 0$} & $-$ &   {\small No}  \\\hline
  {$\E_{6,36}(\alpha) = \E_{11}\oplus N_{5,10}(\alpha)$} &   {$e_{1}^{2} = e_{1},$ $e_{2}^{2} = e_{3} + e_{4},$ $e_{3}^{2} = e_{6},$ $e_{4}^{2} = -e_{6},$ $e_{5}^{2} = \alpha e_{6},$ $e_{6}^{2} = 0$ with $\alpha \in K^{*}$} & $-$ &   {\small No}  \\\hline
  {$\E_{6,37}(\alpha) = \E_{11}\oplus N_{5,11}(\alpha)$} &   {$e_{1}^{2} = e_{1},$ $e_{2}^{2} = e_{3} + e_{4},$ $e_{3}^{2} = e_{6},$ $e_{4}^{2} = -e_{6},$ $e_{5}^{2} = \alpha(e_{3} + e_{4}),$ $e_{6}^{2} = 0$ with $\alpha \in F^{*}$} & $-$ &   {\small No}  \\\hline
  {$\E_{6,38}(\alpha, \beta) = \E_{11}\oplus N_{5,12}(\alpha, \beta)$} &   {$e_{1}^{2} = e_{1},$ $e_{2}^{2} = e_{3} + e_{4},$ $e_{3}^{2} = e_{6},$ $e_{4}^{2} = -e_{6},$ $e_{5}^{2} = \alpha(e_{3} + e_{4}) + \beta e_{6},$ $e_{6}^{2} = 0$ with $\alpha, \beta \in F^{*}$} & $-$ &   {\small No}  \\\hline
  {$\E_{6,39} = \E_{22}\oplus N_{4,1}$} &   {$e_{1}^{2} = e_{1},$ $e_{2}^{2} = e_{2},$ $e_{3}^{2} = e_{4}^{2} = e_{5}^{2} = e_{6}^{2} = 0$ } & $-$ &   {\small Yes}  \\\hline
  {$\E_{6,40} = \E_{22}\oplus N_{4,2}$} &   {$e_{1}^{2} = e_{1},$ $e_{2}^{2} = e_{2},$ $e_{3}^{2} = e_{4},$ $e_{4}^{2} = e_{5}^{2} = e_{6}^{2} = 0$ } & $-$ &   {\small Yes}  \\\hline
  $\E_{6,41}(\alpha) = \E_{22}\oplus N_{4,3}(\alpha)$ &   {$e_{1}^{2} = e_{1},$ $e_{2}^{2} = e_{2},$ $e_{3}^{2} = e_{5},$ $e_{4}^{2} = \alpha e_{5},$ $e_{5}^{2} = e_{6}^{2} = 0$ with $\alpha \in K^{*}$ } & $-$ &   {\small Yes}  \\\hline
  {$\E_{6,42} = \E_{22}\oplus N_{4,4}$} &   {$e_{1}^{2} = e_{1},$ $e_{2}^{2} = e_{2},$ $e_{3}^{2} = e_{4},$ $e_{4}^{2} = 0,$ $e_{5}^{2} = e_{6},$ $e_{6}^{2} = 0$ } & $-$ &   {\small Yes}  \\\hline
  {$\E_{6,43}(\alpha, \beta) = \E_{22}\oplus N_{4,5}(\alpha, \beta)$} &   {$e_{1}^{2} = e_{1},$ $e_{2}^{2} = e_{2},$ $e_{3}^{2} = e_{6},$ $e_{4}^{2} =  \alpha e_{6},$ $e_{5}^{2} =  \beta e_{6},$ $e_{6}^{2} = 0$ $\alpha, \beta \in K^{*}$ } & $-$ &   {\small Yes}  \\\hline
  {$\E_{6,44} = \E_{22}\oplus N_{4,6}$} &   {$e_{1}^{2} = e_{1},$ $e_{2}^{2} = e_{2},$ $e_{3}^{2} = e_{4} + e_{5},$ $e_{4}^{2} =  e_{6},$ $e_{5}^{2} =  -e_{6},$ $e_{6}^{2} = 0$ } & $-$ &   {\small No}  \\\hline
  {$\E_{6,45} = \E_{33}\oplus N_{3,1}$} &   {$e_{1}^{2} = e_{1},$ $e_{2}^{2} = e_{2},$ $e_{3}^{2} = e_{3},$ $e_{4}^{2} = e_{5}^{2} =  e_{6}^{2} = 0$} & $-$ &   {\small Yes}  \\\hline
  {$\E_{6,46} = \E_{33}\oplus N_{3,2}$} &   {$e_{1}^{2} = e_{1},$ $e_{2}^{2} = e_{2},$ $e_{3}^{2} = e_{3},$ $e_{4}^{2} = e_{5},$ $e_{5}^{2} =  e_{6}^{2} = 0$} & $-$ &   {\small Yes}  \\\hline
  {$\E_{6,47}(\alpha) = \E_{33}\oplus N_{3,3}(\alpha)$} &   {$e_{1}^{2} = e_{1},$ $e_{2}^{2} = e_{2},$ $e_{3}^{2} = e_{3},$ $e_{4}^{2} = e_{6},$ $e_{5}^{2} =  \alpha e_{6},$ $e_{6}^{2} = 0$ with $\alpha \in F^{*}$} & $-$ &   {\small Yes}  \\\hline
  {$\E_{6,48} = \E_{44}\oplus N_{2,1}$} &   {$e_{1}^{2} = e_{1},$ $e_{2}^{2} = e_{2},$ $e_{3}^{2} = e_{3},$ $e_{4}^{2} = e_{4},$ $e_{5}^{2} =  e_{6}^{2} = 0$ } & $-$ &   {\small Yes}  \\\hline
  {$\E_{6,49} = \E_{44}\oplus N_{2,2}$} &   {$e_{1}^{2} = e_{1},$ $e_{2}^{2} = e_{2},$ $e_{3}^{2} = e_{3},$ $e_{4}^{2} = e_{4},$ $e_{5}^{2} =  e_{6},$ $e_{6}^{2} = 0$ } & $-$ &   {\small Yes}  \\\hline
  {$\E_{6,50} = \E_{55}\oplus N_{1,1}$} &   {$e_{1}^{2} = e_{1},$ $e_{2}^{2} = e_{2},$ $e_{3}^{2} = e_{3},$ $e_{4}^{2} = e_{4},$ $e_{5}^{2} =  e_{5},$ $e_{6}^{2} = 0$ } & $-$ &   {\small Yes}  \\\hline
  {$\E_{6,51} = \E_{66}$} &   {$e_{1}^{2} = e_{1},$ $e_{2}^{2} = e_{2},$ $e_{3}^{2} = e_{3},$ $e_{4}^{2} = e_{4},$ $e_{5}^{2} =  e_{5},$ $e_{6}^{2} = e_{6}$ } & $-$ &   {\small Yes}  \\
\hline
\end{longtable}
}
\end{prop}


\end{document}